\documentclass[a4paper,11pt,reqno]{amsart}
\usepackage{amsmath}
\usepackage{amsthm,enumerate}

\usepackage{graphicx}
\usepackage{amssymb}

\usepackage{appendix}

\usepackage[french,italian,english]{babel}

\usepackage[utf8x]{inputenc}
\usepackage{fancyhdr}

\usepackage{calc}
\usepackage{url}

\DeclareMathOperator{\sectional}{sec}
\DeclareMathOperator{\Rico}{Ric_o}

\usepackage[text={6.3in,8.6in},centering]{geometry}

\usepackage{srcltx, inputenc}

\usepackage[T1]{fontenc}

\usepackage[usenames,dvipsnames]{color}

\makeatletter
\def\cleardoublepage{\clearpage\if@twoside \ifodd\c@page\else%
         \hbox{}%
     \thispagestyle{empty}
     \newpage%
     \if@twocolumn\hbox{}\newpage\fi\fi\fi}
\makeatother

\theoremstyle{plain}
\newtheorem{theorem}{Theorem}[section]

\newtheorem{definition}[theorem]{Definition}

\newtheorem{lemma}[theorem]{Lemma}

\newtheorem{proposition}[theorem]{Proposition}
\newtheorem{remark}[theorem]{Remark}

\numberwithin{equation}{section}
\theoremstyle{definition}

\newcommand{\R}{\ensuremath{\mathbb{R}}}

\begin{document}

\title[Smoothing effects and infinite time blowup for reaction-diffusion equations]{Smoothing effects and infinite time blowup\\ for reaction-diffusion equations:\\ an approach via Sobolev and Poincar\'e inequalities}

\author{Gabriele Grillo}
\address{\hbox{\parbox{5.7in}{\medskip\noindent{Dipartimento di Matematica,\\
Politecnico di Milano,\\
   Piazza Leonardo da Vinci 32, 20133 Milano, Italy.
   \\[3pt]
        \em{E-mail address: }{\tt
          gabriele.grillo@polimi.it\\ 
          }}}}}

\author{Giulia Meglioli}
\address{\hbox{\parbox{5.7in}{\medskip\noindent{Dipartimento di Matematica,\\
Politecnico di Milano,\\
   Piazza Leonardo da Vinci 32, 20133 Milano, Italy.
   \\[3pt]
        \em{E-mail address: }{\tt
          giulia.meglioli@polimi.it
          }}}}}

\author{Fabio Punzo}
\address{\hbox{\parbox{5.7in}{\medskip\noindent{Dipartimento di Matematica,\\
Politecnico di Milano,\\
   Piazza Leonardo da Vinci 32, 20133 Milano, Italy. \\[3pt]
        \em{E-mail address: }{\tt
          fabio.punzo@polimi.it}}}}}

\subjclass[2010]{Primary: 35K57. Secondary: 35B44, 58J35, 35K65, 35R01.}
\keywords{Reaction diffusion equations. Riemannian manifolds. Blow-up. Global existence. Diffusions with weights.}

\maketitle

\maketitle              

\begin{abstract} We consider reaction-diffusion equations either posed on Riemannian manifolds or in the Euclidean weighted setting, with pow\-er-type nonlinearity and slow diffusion of porous medium type. We consider the particularly delicate case $p<m$ in problem \eqref{problema}, a case presently largely open even when the initial datum is smooth and compactly supported. We prove global existence for L$^m$ data, and that solutions corresponding to such data are bounded at all positive times with a quantitative bound on their L$^\infty$ norm. We also show that on Cartan-Hadamard manifolds with curvature pinched between two strictly negative constants, solutions corresponding to sufficiently large L$^m$ data give rise to solutions that blow up pointwise everywhere in infinite time, a fact that has no Euclidean analogue. The methods of proof are functional analytic in character, as they depend solely on the validity of the Sobolev and of the Poincar\'{e} inequalities.
As such, they are applicable to different situations, among which we single out the case of (mass) weighted reaction-diffusion equation in the Euclidean setting. In this latter setting we also consider, with stronger results for large times, the case of globally integrable weights.
\end{abstract}

\bigskip

\selectlanguage{french}
\begin{abstract}
Nous considérons des équations de réaction-diffusion posées soit sur des variétés Riemanniennes, soit dans le cadre Euclidien pondéré, avec une non-linéarité de type puissance et une diffusion lente de type moyen poreux. Nous considérons le cas particulièrement délicat $p<m$ dans le problème \eqref{problema}, un cas actuellement ouvert même lorsque la donnée initiale est régulière et à support compact. Nous prouvons l'existence globale pour données L$^m$, et que les solutions corréspondantes à ces données sont bornées à tout instant positif par une borne quantitative de leur norme L$^\infty$. Nous montrons également que sur des variétés de Cartan-Hadamard dont la courbure est comprise entre deux constantes strictement négatives, les solutions corréspondantes à des données L$^m$ suffisamment grandes donnent lieu à des solutions qui explosent en tous les points dans un temps infini, un fait qui n'a pas d'analogue Euclidien. Les méthodes de preuve sont d'analyse fonctionnelle, car elles dépendent uniquement de la validité des inégalités de Sobolev et de Poincar\'{e}. 
En tant que telles, elles sont applicables à différentes situations, parmi lesquelles nous soulignons le cas de l'équation de réaction-diffusion pondérée (en masse) dans le cadre Euclidien. Dans ce dernier cas, nous considérons également, avec des résultats plus forts pour les périodes plus longues, le cas des pondérations intégrables globalement.
\end{abstract}
\bigskip
\bigskip

\selectlanguage{english}

\section{Introduction}

Let $M$ be a complete noncompact Riemannian manifold of infinite volume, whose dimension $N$ will be required throughout the paper to satisfy the bound $N\geq 3$. 
Let us consider the following Cauchy problem, for any $T>0$
\begin{equation}\label{problema}
\begin{cases}
\, u_t= \Delta u^m +\, u^p & \text{in}\,\, M\times (0,T) \\
\,\; u =u_0 &\text{in}\,\, M\times \{0\}
\end{cases}
\end{equation}
where $\Delta$ is the Laplace-Beltrami operator. We shall assume throughout this paper that $1<p\,<\,m$ and that the initial datum $u_0$ is nonnegative.
We let L$^q(M)$ be as usual the space of those measurable functions $f$ such that $|f|^q$ is integrable w.r.t. the Riemannian measure $\mu$ and make
the following basic assumptions on $M$, which amount to assuming the validity of both the Poincar\'{e} and the Sobolev inequalities on $M$:
\begin{equation}\label{P}
(\textrm{Poincar\'e\ inequality)}\ \ \ \ \ \|v\|_{L^2(M)} \le \frac{1}{C_p} \|\nabla v\|_{L^2(M)} \quad \text{for any}\,\,\, v\in C_c^{\infty}(M);
\end{equation}
\begin{equation}\label{S}
(\textrm{Sobolev\ inequality)}\ \ \ \ \ \ \|v\|_{L^{2^*}(M)} \le \frac{1}{C_s} \|\nabla v\|_{L^2(M)}\quad \text{for any}\,\,\, v\in C_c^{\infty}(M),
\end{equation}
where $C_p$ and $C_s$ are numerical constants and $2^*:=\frac{2N}{N-2}$. The validity of \eqref{P}, \eqref{S} puts constraints on $M$, and we comment that it is e.g. well known that, on \it Cartan-Hadamard manifolds\rm, namely complete and simply connected manifolds that have everywhere non-positive sectional curvature, \eqref{S} always holds. Furthermore, when $M$ is Cartan-Hadamard and, besides, $\sectional\le -c<0$ everywhere, $\sectional$ indicating sectional curvature, it is known that \eqref{P} holds as well, see e.g. \cite{Grig, Grig3}. Thus, both \eqref{P}, \eqref{S} hold when $M$ is Cartan-Hadamard and sec$\,\le -c<0$ everywhere, a case that strongly departs from the Euclidean situation but covers a wide class of manifolds, including e.g. the fundamental example of the hyperbolic space $\mathbb{H}^n$, namely that Cartan-Hadamard manifold whose sectional curvatures equal -1 everywhere (or the similar case in which $\sectional=-k$ everywhere, for a given $k>0$).

The behaviour of solutions to \eqref{problema} is influenced by competing phenomena. First of all there is a diffusive pattern associated with the so-called \it porous medium equation\rm, namely the equation
\begin{equation}\label{pme}
u_t \,=\, \Delta u^m \quad \textrm{in}\;\; M\times (0,T)\,,
\end{equation}
where the fact that we keep on assuming $m>1$ puts us in the \it slow diffusion case\rm. It is known that when $M={\mathbb R}^n$ and, more generally, e.g. when $M$ is a Cartan-Hadamard manifold, solutions corresponding to compactly supported data have compact support for all time, in contrast with the properties valid for solutions to the heat equation, see \cite{V}. But it is also well-known that, qualitatively speaking, \it negative curvature accelerates diffusions\rm, a fact that is apparent first of all from the behaviour of solutions of the classical heat equation. In fact, it can be shown that the standard deviation of a Brownian particle on the hyperbolic space $\mathbb{H}^n$ behaves \it linearly \rm in time, whereas in the Euclidean situation it is proportional to $\sqrt t$. Similarly, the heat kernel decays exponentially as $t\to+\infty$ whereas one has a power-type decay in the Euclidean situation.

In the Riemannian setting the study of \eqref{pme} has started recently, see e.g. \cite{GIM}, \cite{GMhyp}, \cite{GM2}, \cite{GMPbd}, \cite{GMPrm}, \cite{GMV}, \cite{Pu1},  \cite{VazH}, noting that in some of those papers also the case $m<1$ in \eqref{pme}, usually referred to as the \it fast diffusion \rm case, is studied. Nonlinear diffusion gives rise to speedup phenomena as well. In fact, considering again the particularly important example of the hyperbolic space  $\mathbb{H}^n$ (cf. \cite{VazH}, \cite{GM2}), the $L^\infty$ norm of a solution to \eqref{pme} satisfies $\|u(t)\|_\infty\asymp \left(\frac{\log t}t\right)^{1/(m-1)}$ as $t\to+\infty$, a time decay which is \it faster \rm than the corresponding Euclidean bound. Besides, if the initial datum is compactly supported, the volume $\mathsf{V}(t)$ of the support of the solution $u(t)$ satisfies $\mathsf{V}(t)\asymp t^{1/(m-1)}$ as $t\to+\infty$, while in the Euclidean situation one has $\mathsf{V}(t)\asymp t^{\beta(N,m)}$ with $\beta(N,m)<1/(m-1)$.

The second driving factor influencing the behaviour of solutions to \eqref{problema} is the \it reaction term \rm $u^p$, which has the positive sign and, thus, might drive solutions towards blow-up. This kind of problems has been widely studied in the Euclidean case $M= {\mathbb R}^N$, especially in the case $m=1$ (linear diffusion). The literature for this problem is huge and there is no hope to give a comprehensive review here, hence we just mention that blow-up occurs for all nontrivial nonnegative data when $p\le1+2/N$, while global existence prevails for $p>1+ 2/N$ (for specific  results see e.g. \cite{CFG}, \cite{DL}, \cite{F}, \cite{FI}, \cite{H}, \cite{L}, \cite{Q}, \cite{S}, \cite{W}, \cite{Y}). On the other hand, it is known that when $M=\mathbb{H}^N$ and $m=1$, for all $p>1$ and sufficiently small nonnegative data there exists a global in time solution, see \cite{BPT}, \cite{WY}, \cite{WY2}, \cite{Pu3}.

As concerns the slow diffusion case $m>1$, in the Euclidean setting it is shown in \cite{SGKM} that, when the initial datum is nonnegative, nontrivial and compactly supported, for any $p>1$, all sufficiently large data give rise to solutions blowing up in finite time. Besides, if $p\in\left(1,m+\frac2N\right)$, \it all \rm such solutions blow up in finite time. Finally,
if $p>m+\frac2N$, all sufficiently small data give rise to global solutions. For subsequent, very detailed results e.g. about the type of possible blow-up and, in some case, on continuation after blow-up, see \cite{GV}, \cite{MQV}, \cite{Vaz1} and references quoted therein.

In the Riemannian setting, existence of global solutions and blow-up in finite time for problem \eqref{problema} have been first studied in \cite{Z}, under the assumption that the volume of geodesic balls of radius $R$ grows as $R^{\alpha}$ with $\alpha\geq 2$; this kind of assumption is typically associated to \it nonnegative \rm curvature, thus the opposite situation w.r.t. the one we are studying here, in which the volume of geodesic balls grows at least exponentially as a function of the radius $R$. The results in the setting studied in \cite{Z} are qualitatively similar to the Euclidean ones.

The situation on negatively curved manifolds is significantly different, and the first results in this connection have been shown in \cite{GMPv}, where only the case of nonnegative, compactly supported data is considered. Among the results of that paper, we mention the case that a \it dichotomy phenomenon \rm holds when $p>m$, in the sense that under appropriate curvature conditions, compatible with the assumptions made in the present paper,  all sufficiently small data give rise to solutions existing globally in time, whereas sufficiently large data give rise to solutions blowing up in finite time. Results were only partial when $p<m$, since it has been shown that when $p\in\left(1,\frac{1+m}{2}\right]$ and again under suitable curvature conditions, all solutions corresponding to compactly supported initial data exist globally in time, and blow up everywhere pointwise in infinite time. When $p\in\left(\frac{1+m}{2},m\right)$, precise information on the asymptotic behaviour is not known, since blowup is shown to occur at worse in infinite time, but could in principle occur before.

\subsection {Qualitative statements of main results in the manifold setting}
\medskip
We extend here the results of \cite{GMPv} in two substantial aspects. In fact, we summarize our main results as follows.

\begin{itemize}
\item The methods of \cite{GMPv} rely heavily on explicit \it barrier arguments\rm, that by their very same nature are applicable to compactly supported data only and, in addition, require explicit curvature bounds in order to be applicable. We prove here global existence for L$^m$ data and prove \it smoothing effects \rm for solutions to \eqref{problema}, where by smoothing effect we mean the fact that L$^m$ data give rise to global solutions $u(t)$ such that $u(t)\in \text{L}^\infty$ for all $t>0$, with quantitative bounds on their L$^\infty$ norm. This will be a consequence \it only \rm of the validity of Sobolev and Poincar\'e inequalities \eqref{S}, \eqref{P}, see Theorem \ref{teoesistenza}.

\item As a consequence, combining this fact with some results proved in \cite{GMPv}, we can prove that, on manifolds satisfying e.g. $-c_1\le \textrm{sec}\le -c_2$ with $c_1\ge c_2>0$, thus encompassing the particularly important case of the hyperbolic space $\mathbb{H}^n$ (somewhat weaker lower curvature bounds can be assumed), any solution $u(t)$ to \eqref{problema} corresponding to an initial datum $u_0\in\text{L}^m$ exists globally  and, provided $u_0$ is sufficiently large, it satisfies the property
    \[
    \lim_{t\to+\infty} u(x, t)=+\infty\ \ \ \forall x\in M,
    \]
    namely \it complete blowup in infinite time \rm occurs for such solutions to \eqref{problema} in the whole range $p\in(1,m)$, see Theorem \ref{blowup}.

\end{itemize}

Our results can also be seen as an extension of some of the results proved in \cite{Sacks}. However, the proof of the smoothing estimate given in \cite[Theorem 1.3]{Sacks} is crucially based on the assumption that the measure of the domain where the problem is posed is finite. This is not true in our setting. So, even if we use some general idea introduced in \cite{Sacks}, our proofs and results are in general quite different from those in \cite{Sacks}.

For detailed reference to smoothing effect for linear evolution equations see \cite{D}, whereas we refer to \cite{Vsmooth} for a general treatment of smoothing effects for nonlinear diffusions, and to \cite{BG, GMPo,GM2} for connections with functional inequalities in the nonlinear setting.

We mention phenomena similar to the ones discussed in the present paper occur in qualitatively related but different settings. For example, we mention that solutions to the heat equation with Dirichlet boundary conditions in a twisted tube (namely a straight tube in $\mathbb{R}^3$ whose cross-section is twisted in a given compact region) give rise to smoothing estimates that are \it stronger \rm for large times than the ones corresponding to the untwisted situation, i.e. the geometry improves the smoothing effects, see \cite{KZ1,KZ2,GKP}.

\subsection{Qualitative statements of main results for Euclidean, weighted reaction-diffusion equations}

\bigskip The main result given in Theorem \ref{teoesistenza} depend essentially only on the validity of inequalities \eqref{P} and \eqref{S}, and as such is almost immediately generalizable to different contexts. As a particularly significant situation, we single out the case of Euclidean, mass-weighted reaction diffusion equations. In fact we consider the problem
\begin{equation}\label{problema2}
\begin{cases}
\rho\, u_t= \Delta u^m +\rho\, u^p & \text{in}\,\, \R^N\times (0,T) \\
u\,\,  =u_0 &\text{in}\,\, \R^N\times \{0\},
\end{cases}
\end{equation}
in the Euclidean setting, where $\rho:\R^N\to\R$ is strictly positive, continuous and  bounded, and represents a \it mass density \rm. The problem is naturally posed in the weighted spaces
$$L^q_{\rho}(\R^N)=\left\{v:\R^N\to\R\,\, \text{measurable}\,\,  ,   \,\, \|v\|_{L^q_{\rho}}:=\left(\int_{\R^N} \,v^q\rho(x)\,dx\right)^{1/q}<+\infty\right\},$$

This kind of models originates in a physical model provided in \cite{KR}. There are choices of $\rho$ ensuring that the following analogues of \eqref{P} and \eqref{S} hold:
\begin{equation}\label{P-pesi}
\|v\|_{L^2_{\rho}(\R^N)} \le \frac{1}{C_p} \|\nabla v\|_{L^2(\R^N)} \quad \text{for any}\,\,\, v\in C_c^{\infty}(\R^N)
\end{equation}
and
\begin{equation}\label{S-pesi}
\|v\|_{L^{2^*}_{\rho}(\R^N)} \le \frac{1}{C_s} \|\nabla v\|_{L^2(\R^N)}\quad \text{for any}\,\,\, v\in C_c^{\infty}(\R^N)
\end{equation}
for suitable positive constants. In fact, in order to make a relevant example, if $\rho(x)\asymp |x|^{-a}$ for a suitable $a>0$, it can be shown that \eqref{P-pesi} holds if $a\ge2$ (see e.g. \cite{GMPo} and references therein), whereas also \eqref{S-pesi} is obviously true for any $a>0$ because of the validity of the usual, unweighted Sobolev inequality and of the assumptions on $\rho$. Of course more general cases having a similar nature but where the analogue of \eqref{S-pesi} is not a priori trivial, could be considered, but we focus on that example since it is widely studied in the literature and because of its physical significance.

In \cite{MT, MTS} a large class of nonlinear reaction-diffusion equations, including in particular problem \eqref{problema2} under certain conditions on $\rho$, is investigated. It is proved that a global solution exists, (see \cite[Theorem 1]{MT}) provided that $\rho(x)=|x|^{-a}$ with $a\in (0,2)$,
\[p>m+\frac{2-a}{N-a},\]
and $u_0\geq 0$ is small enough. In addition, a smoothing estimate holds. On the other hand, if $\rho(x)=|x|^{-a}$
 or $\rho(x)=(1+|x|)^{-a}$ with $a\in [0,2)$, $u_0\not\equiv 0$ and
 \[1<p<m+\frac{2-a}{N-a},\] then any nonnegative solution blows up in a suitable sense.
Such results have also been generalized to more general initial data, decaying at infinity with a certain rate (see \cite{MTS}).
Finally, in \cite[Theorem 2]{MT}, it is shown that if $p>m$, $\rho(x)=(1+|x|)^{-a}$ with $a>2$, and $u_0$ is small enough, a global solution exists.

Problem \eqref{problema2} has also been studied in \cite{MP1}, \cite{MP2}, by constructing and using suitable barriers, initial data being continuous and compactly supported. In particular, in \cite{MP1} the case that $\rho(x)\asymp |x|^{-a}$ for $|x|\to+\infty$
with $a\in (0,2)$ is addressed. It is proved that for any $p>1$, if $u_0$ is large enough, then blowup occurs. On the other hand, if $p>\bar p$, for a certain
$\bar p>m$ depending on $m, p$ and $\rho$, and $u_0$ is small enough, then global existence of bounded solutions prevails. Moreover, in \cite{MP2} the case that $a\geq 2$ is investigated. For $a=2$, blowup is shown to occur when $u_0$ is big enough, whereas global existence holds when $u_0$ is small enough. For $a>2$ it is proved that if $p>m$, $u_0\in L^{\infty}_{\rm{loc}}(\mathbb R^N)$ and goes to $0$ at infinity with a suitable rate, then there exists a global bounded solution. Furthermore, for the same initial datum $u_0$, if $1<p<m$, then there exists a global solution, which could blow up as $t\to +\infty$\,.

\smallskip

 Our main results in this setting can be summarized as follows.

\begin{itemize}
\item We prove in Theorem \ref{teoesistenza2} global existence and smoothing effects for solutions to \eqref{problema2}, assuming that the weight $\rho:\R^N\to\R$ is strictly positive, smooth and  bounded, so that \eqref{S-pesi} necessarily holds, and assuming the validity of \eqref{P-pesi}. In particular, L$^m$ data give rise to global solutions $u(t)$ such that $u(t)\in $\,L$^\infty$ for all $t>0$, with quantitative bounds on their L$^\infty$ norm. By constructing a specific, delicate example, we show in Proposition \ref{teosubsolutioncritical} that the bound on the L$^\infty$ norm (which involves a quantity diverging as $t\to+\infty$) is qualitatively sharp, in the sense that there are examples of weights for which our running assumption holds and for which blow-up of solutions in infinite time holds pointwise everywhere (we refer to this property by saying that \it complete blowup in infinite time \rm occurs). We also prove, by similar methods which follow the lines of \cite{Sacks}, different smoothing effects which are stronger for large times, when $\rho$ is in addition assumed to be integrable, see Theorem \ref{teoesistenza3}.
\end{itemize}

Let us mention that the results in \cite{MP2} for $1<p<m$ are improved here  in various directions. In fact, now we consider a larger class of initial data $u_0$, since
we do not require that they are locally bounded; moreover, in \cite{MP2} no smoothing estimates are addressed. Furthermore, the fact that for integrable weights $\rho$ we have global existence of bounded solutions does not have
a counterpart in \cite{MP2}, nor has the blowup results in infinite time.

\subsection{On some open problems} As stated above, the present paper settles the problem of global existence of solutions to problem \eqref{problema} on manifolds $M$ supporting both the Sobolev and the Poincaré inequalities, in the case $1<p<m$ and for data belonging to $L^m(M)$. It is also shown that solutions corresponding to such data are bounded for all $t>0$, with quantitative bounds on the $L^\infty(M)$ norm of solutions for all $t>0$. We also settle the long-time behaviour of solutions to problem \eqref{problema} on manifolds $M$ whose curvature is pinched between two strictly negative constants, where $1<p<m$ and data belong to $L^m(M)$, showing that they blowup pointwise in infinite time. The following questions are however open for further investigation:

\begin{itemize}
\item Does similar results hold for data in Lebesgue spaces $L^q(M)$ with $q\not=m$? The present method of proof does not extend to such data.
\item Do \it all \rm initial data in $L^q(M)$ blow up in infinite time, or the long-time asymptotic of small data is different?
\end{itemize}

\subsection{Organizazion of the paper} In Section \ref{statements} we collect the relevant definitions and state our main results, both in the setting of Riemannian manifolds and in the Euclidean, weighted case. In Section \ref{elliptic} we prove some crucial results for an auxiliary elliptic problem, that will then be used in Section \ref{Lp} to show bounds on the $\text{L}^p$ norms of solutions to certain evolution problems posed on geodesic balls. In Section \ref{proofs} we conclude the proof of our main results for the case of reaction-diffusion problems on manifolds. In Section \ref{weights} we briefly comment on the adaptation to be done to deal with the weighted Euclidean case, and prove the additional results valid in the case of an integrable weight. We also discuss there a delicate example showing that complete blowup in infinite time may occur under the running assumptions.

\section{Preliminaries and statement of main results}\label{statements}

We first define the concept of solution to \eqref{problema} that we shall use hereafter. It will be meant in the very weak, or distributional, sense.

\begin{definition}\label{1}
Let $M$ be a complete noncompact Riemannian manifold of infinite volume. Let $1<p<m$ and $u_0\in{\textrm L}^m(M)$, $u_0\ge0$. We say that the function $u$ is a solution to problem \eqref{problema} in the time interval $[0,T)$ if
$$
u\in L^m(M\times(0,T))\,,
$$
and for any $\varphi \in C_c^{\infty}(M\times[0,T])$ such that $\varphi(x,T)=0$ for any $x\in M$, $u$ satisfies the equality:
\begin{equation*}
\begin{aligned}
-\int_0^T\int_{M} \,u\,\varphi_t\,d\mu\,dt =&\int_0^T\int_{M} u^m\,\Delta\varphi\,d\mu\,dt\,+ \int_0^T\int_{M} \,u^p\,\varphi\,d\mu\,dt \\
& +\int_{M} \,u_0(x)\,\varphi(x,0)\,d\mu.
\end{aligned}
\end{equation*}
\end{definition}

\begin{theorem}\label{teoesistenza}
Let $M$ be a complete, noncompact manifold of infinite volume such that the Poincaré and Sobolev inequalities \eqref{P} and \eqref{S} hold on $M$. Let $1<p<m$ and $u_0\in{\textrm L}^m(M)$, $u_0\ge0$. Then problem \eqref{problema} admits a solution for any $T>0$,  in the sense of Definition \ref{1}. Moreover for any $T>\tau>0$ one has $u\in L^{\infty}(M\times(\tau,T))$ and there exist numerical constants $c_1, c_2>0$, independent of $T$, such that, for all $t>0$ one has

\begin{equation}\begin{aligned}\label{smoothing}
\|u(t)\|_{L^{\infty}(M)}&\le c_1e^{c_2t}\left\{ \|u_{0}\|_{L^m(M)}^{\frac{2m}{2m +N(m-p)}}+\frac{\|u_{0}\|_{L^m(M)}^{\frac{2m}{2m+N(m-1)}}}{t^{\frac{N}{2m+N(m-1)}}}\right\}.
\end{aligned}\end{equation}

Besides, if $q>1$ and $u_0\in L^q(M)\cap L^m(M)$, then there exists $C(q)>0$ such that \begin{equation}\label{Lq}
\|u(t)\|_{L^q(M)}\le  e^{C(q)t}\|u_{0}\|_{L^q(M)}\quad \textrm{for all }\,\, t>0\,.
\end{equation}

\end{theorem}

One may wonder whether the upper bound in \eqref{smoothing} is qualitatively sharp, since its r.h.s. involves a function of time that tends to $+\infty$ as $t\to+\infty$. This is indeed the case, since there is a wide class of situations covered by Theorem \ref{teoesistenza} in which classes of solutions do indeed satisfy $\|u(t)\|_\infty\to+\infty$ as $t\to+\infty$ and show even the much stronger property of \it blowing up pointwise everywhere in infinite time\rm. In fact, as a direct consequence of Theorem \ref{teoesistenza}, of known geometrical conditions for the validity of \eqref{P} and \eqref{S}, and of some results given in \cite{GMPv}, we can prove the following result. We stress that this property has no Euclidean analogue for the corresponding reaction-diffusion problem.

\begin{theorem}\label{blowup}
Let $M$ be a Cartan-Hadamard manifold and let $\sectional$ denote sectional curvature, $\Rico$ denote the Ricci tensor in the radial direction with respect to a given pole $o\in M$. Assume that the following curvature bounds hold everywhere on $M$, for suitable $k_1\ge k_2>0$:

\begin{equation*}
\Rico(x)\geq - k_1;\ \ \ \sectional\le -k_2\,.
\end{equation*}
Then the results of Theorem \ref{teoesistenza} hold. Besides, consider any nonnegative solution $u$ to \eqref{problema} corresponding to an initial datum $u_0\in{\textrm L}^m(M)$ which is sufficiently large in the sense that $u_0\ge v_0$ for a suitable nonnegative and sufficiently large function $v_0\in C_c^{0}(M)$. Then $u$
satisfies
\[
\lim_{t\to+\infty}u(x,t)=+\infty\ \ \ \forall x\in M.
\]
\end{theorem}

Observe that, as it will appear from the proof, for the function $v_0$ in Theorem \ref{blowup} we require that $v_0>0$ in a geodesic ball $B_R$ with $R>0$ and $m:=\inf_{B_R}v_0$ both sufficiently large.

\subsection{Weighted reaction-diffusion equations in the Euclidean space}\label{weight}

As mentioned in the Introduction, the methods used in proving Theorem \ref{teoesistenza} are general enough, being based on functional inequalities only, to be easily generalized to different contexts. We single out here the one in which reaction-diffusion equations are considered in the Euclidean setting, but in which diffusion takes place in a medium having a nonhomogeneous density, see e.g. \cite{KR}, \cite{MT}, \cite{MTS}, \cite{MTS2} and references quoted therein.

We consider a \it weight \rm $\rho:\R^N\to\R$ such that
\begin{equation}\label{rho2}
\rho \in C(\R^N)\cap L^{\infty}(\R^N), \ \ \rho(x)>0 \,\, \text{for any}\,\, x\in \R^N,
\end{equation}
and the associated weighted Lebesgue spaces $$L^q_{\rho}(\R^N)=\{v:\R^N\to\R\,\, \text{measurable}\,\,  |   \,\, \|v\|_{L^q_{\rho}}<+\infty\},$$
where $ \|v\|_{L^q_{\rho}}:=\int_{\R^N} \rho(x)\,|v(x)|^q\,dx.$
Moreover, we assume that $\rho$ is such that the weighted Poincar\'{e} inequality \eqref{P-pesi} holds. By construction and by the  assumptions in \eqref{rho2} it follows that the weighted Sobolev inequality \eqref{S-pesi} also holds, as a consequence of the usual Sobolev inequality in $\R^N$ and of \eqref{rho2}.

Moreover, we let $u_0:\R^N\to\R$ be such that
\begin{equation*}
\,\,\, u_0\in L^m_{\rho}(\R^N),\,\,\,\,\, u_0(x)\ge 0 \,\, \text{for a.e.}\,\, x\in \R^N
\end{equation*}
and consider, for any $T>0$ and for any $1\,<\,p\,<\,m,$ problem \eqref{problema2}.

The definition of solution we use will be again the very weak one, adapted to the present case.

\begin{definition}\label{2}
Let $1<p<m$ and $u_0\in{\textrm L}_{\rho}^m(\mathbb R^N)$, $u_0\ge0$. Let the weight $\rho$ satisfy \eqref{rho2}. We say that the function $u$ is a solution to problem \eqref{problema2} in the interval $[0, T)$ if
$$
u\in L^m_{\rho}(\mathbb R^N\times(0,T))\,\,\,
$$
and for any $\varphi \in C_c^{\infty}(\mathbb R^N\times[0,T])$ such that $\varphi(x,T)=0$ for any $x\in \mathbb R^N$, $u$ satisfies the equality:
\begin{equation}\label{3a}
\begin{aligned}
-\int_0^T\int_{\mathbb{R}^N} \,u\,\varphi_t\,\rho(x)\,dx\,dt =&\int_0^T\int_{\mathbb R^N} u^m\,\Delta \varphi\,dx\,dt\,+ \int_0^T\int_{\mathbb R^N} \,u^p\,\varphi\,\rho(x)\,dx\,dt \\
& +\int_{\mathbb R^N} \,u_0(x)\,\varphi(x,0)\,\rho(x)\,dx.
\end{aligned}
\end{equation}
\end{definition}

\begin{theorem}\label{teoesistenza2}

Let $\rho$ satisfy \eqref{rho2} and assume that the weighted Poincaré inequality \eqref{P-pesi} holds. Let $1<p<m$ and $u_0\in{\textrm L}_{\rho}^m(\mathbb R^N),$ $u_0\ge0$. Then problem \eqref{problema2} admits a solution for any $T>0$,  in the sense of Definition \ref{2}. Moreover for any $T>\tau>0$ one has $u\in L^{\infty}(\mathbb R^N\times(\tau,T))$ and there exist numerical constants $c_1$, $c_2>0$, independent of $T$, such that, for all $t>0$ one has

\begin{equation}\begin{aligned}\label{smoothingweight}
\|u(t)\|_{L^{\infty}(\R^N)}&\le c_1e^{c_2t}\left\{ \|u_{0}\|_{L^m_{\rho}(\R^N)}^{\frac{2m}{2m +N(m-p)}}+\frac{\|u_{0}\|_{L^m_{\rho}(\R^N)}^{\frac{2m}{2m+N(m-1)}}}{t^{\frac{N}{2m+N(m-1)}}}\right\}.
\end{aligned}\end{equation}

Besides, if $q>1$ and $u_0\in L^q_{\rho}(\mathbb R^N)\cap L^m_{\rho}(\mathbb R^N)$, then there exists $C(q)>0$ such that \begin{equation*}
\|u(t)\|_{L^q_{\rho}(\mathbb R^N)}\le  e^{C(q)t}\|u_{0}\|_{L^q_{\rho}(\mathbb R^N)}\quad \textrm{ for all }\,\, t>0\,.
\end{equation*}

Finally, there are examples of weights satisfying the assumptions of the present Theorem and such that sufficiently large initial data $u_0$ give rise to solutions $u(x,t)$ blowing up pointwise everywhere in infinite time, i.e. such that $\lim_{t\to+\infty}u(x,t)=+\infty$ for all $x\in \mathbb R^N$, so that in particular $\|u(t)\|_\infty\to+\infty$ as $t\to+\infty$ and hence the upper bound in \eqref{smoothingweight} is qualitatively sharp. One can take e.g. $\rho\asymp |x|^{-2}$ as $|x|\to+\infty$ for this to hold.

\end{theorem}

In the case of \it integrable weights \rm one can adapt the methods of \cite{Sacks} to prove a stronger result.

\begin{theorem}\label{teoesistenza3}
Let $\rho$ satisfy \eqref{rho2} and $\rho\in L^1(\mathbb R^N)$. 
Let $1<p<m$ and $u_0\in{\textrm L}_{\rho}^1(\mathbb R^N)$, $u_0\ge0$. Then problem \eqref{problema2} admits a solution for any $T>0$,  in the sense of Definition \ref{2}. Moreover for any $T>\tau>0$ one has $u\in L^{\infty}(\mathbb R^N\times(\tau,T))$ and there exists $C=C(m,p,N,\|\rho\|_{L^1(\mathbb R^N)})>0$, independent of the initial datum $u_0$, such that, for all $t>0$, one has
\begin{equation}\label{absolute}
\|u(t)\|_{L^{\infty}(\mathbb R^N)}\le C \left\{1+\left[\frac{1}{(m-1)t}\right]^{\frac{1}{m-1}}  \right\}.
\end{equation}

\begin{remark}
\rm \begin{itemize}
\item The bound \eqref{absolute} cannot be replaced by a similar one in which the r.h.s. is replaced by $\frac{C}{(m-1)t}$, that would entail $\|u(t)\|_\infty\to0$ as $t\to+\infty$, as customary e.g. in the case of solutions to the Porous Medium Equation posed in bounded, Euclidean domains (see \cite{V}). In fact, it is possible that \it stationary, bounded solutions \rm to \eqref{problema2} exist, provided a positive bounded solution $U$ to the equation
\begin{equation}\label{nonlin}
-\Delta U=\rho U^a
\end{equation}
exists, where $a=p/m<1$. If this fact holds, $V:=U^\frac1m$ is a stationary, bounded, positive solution to the differential equation in \eqref{problema2}, whose $L^\infty$ norm is of course constant in time. In turn, a celebrated results of \cite{BK} entails that positive, bounded solutions to \eqref{nonlin} exist if e.g. $\rho\asymp |x|^{-2-\epsilon}$ for some $\epsilon>0$ as $|x|\to+\infty$ (in fact, a full characterization of the weights for which this holds is given in \cite{BK}), a condition which is of course compatible with the assumptions of Theorem \ref{teoesistenza3}.
\item Of course, the bound \eqref{smoothingweight}, which gives stronger information when $t\to0$, continues to hold under the assumptions of Theorem \ref{teoesistenza3}.
\end{itemize}
\end{remark}

\end{theorem}

\section{Auxiliary results for elliptic problems}\label{elliptic}

Let $x_0,x \in M$ be given. We denote by $r(x)=\textrm{dist}\,(x_0,x)$ the Riemannian distance between $x_0$ and $x$. Moreover, we let $$B_R(x_0):=\{x\in M, \textrm{dist}\,(x_0,x)<R\}$$ be the geodesics ball with center $x_0 \in M$ and radius $R > 0$. Let $x_0\in M$ any fixed reference point. We set $B_R\equiv B_R(x_0)\,.$ As mentioned above, we denote by $\mu$ the Riemannian measure on $M$.

For any given function $v$, we define for any $k\in\R^+$
\begin{equation*}
T_k(v):=\begin{cases} &k\quad \text{if}\,\,\, v\ge k \\ &v \quad \text{if}\,\,\, |v|< k \\ &-k\quad \text{if}\,\,\, v\le -k\end{cases}\,\,.
\end{equation*}

For every $R>0, k>0,$ consider the problem
\begin{equation}\label{problemapalla}
\begin{cases}
\, u_t= \Delta u^m +\, T_k(u^p) & \text{in}\,\, B_R\times (0,+\infty) \\
u=0 &\text{in}\,\, \partial B_R\times (0,+\infty)\\
u=u_0 &\text{in}\,\, B_R\times \{0\}, \\
\end{cases}
\end{equation}
where $u_0\in L^\infty(B_R), u_0\geq 0$.
Solutions to problem \eqref{problemapalla} are meant in the weak sense as follows.

\begin{definition}\label{4}
Let $p<m$. Let $u_0\in L^\infty(B_R), u_0\geq 0$. We say that a nonnegative function $u$ is a solution to problem \eqref{problemapalla} if
$$
u\in L^{\infty}(B_R\times(0,+\infty)), u^m\in L^2\big((0, T); H^1_0(B_R)\big) \quad \textrm{ for any }\, T>0,
$$
and for any $T>0, \varphi \in C_c^{\infty}(B_R\times[0,T])$ such that $\varphi(x,T)=0$ for every $x\in B_R$, $u$ satisfies the equality:
\begin{equation*}
\begin{aligned}
-\int_0^T\int_{B_R} \,u\,\varphi_t\,d\mu\,dt =&- \int_0^T\int_{B_R} \langle \nabla u^m, \nabla \varphi \rangle \,d\mu\,dt\,+ \int_0^T\int_{B_R} \,T_k(u^p)\,\varphi\,d\mu\,dt \\
& +\int_{B_R} \,u_0(x)\,\varphi(x,0)\,d\mu.
\end{aligned}
\end{equation*}
\end{definition}

We also consider elliptic problems of the type
\begin{equation}\label{pbella}
\begin{cases}
-\Delta u &= f \quad \textrm{ in }\,\, B_R\\
\;\quad u & = 0 \quad \textrm{ in }\,\, \partial B_R\,,
\end{cases}
\end{equation}
with $f\in L^q(B_R)$ for some $q>1$.

\begin{definition}\label{defpbell}
We say that $u\in H^1_0(B_R), u\geq 0$ is a weak subsolution to problem \eqref{pbella} if
\[\int_{B_R}\langle \nabla u, \nabla \varphi \rangle\, d\mu \leq \int_{B_R} f\varphi\, d\mu,\]
for any $\varphi\in H^1_0(B_R), \varphi\geq 0$\,.
\end{definition}

The following proposition contains an estimate in the spirit of the celebrated $L^\infty$ estimate of Stampacchia (see, e.g., \cite{KS}, \cite{BC} and references therein). However, the obtained bound and the proof are different. This is due to the fact that we need an estimate independent of the measure of $B_R$, in order to let $R\to +\infty$ when we apply such estimate in the proof of global existence for problem \eqref{problema} (see Remark \ref{remark2} below). Indeed recall that, obviously, since $M$ has infinite measure, $\mu(B_R)\to +\infty$ as $R\to +\infty$.

\begin{proposition}\label{prop1}
Let $f_1\in L^{m_1}(B_R)$ and $f_2\in L^{m_2}(B_R)$ where $m_1>\frac N 2$, $m_2 >\frac{N}{2}\,.$  Assume that $v\in H_0^1(B_R)$, $v\ge 0$ is a subsolution to problem
\begin{equation}\label{25}
\begin{cases}
-\Delta v = (f_1+f_2) & \text{in}\,\, B_R\\
v=0 &\text{on}\,\, \partial B_R
\end{cases}.
\end{equation}
in the sense of Definition \ref{defpbell}. Let $\bar k>0$. Then

\begin{equation}\label{eqa3}
\|v\|_{L^{\infty}(B_R)}\le \left\{C_1\|f_1\|_{L^{m_1}(B_R)}+ C_2\|f_2\|_{L^{m_2}(B_R)} \right\}^{\frac{1}{s}} \|v\|_{L^{1}(B_R)}^{\frac{s-1}{s}}+\bar k,
\end{equation}
where
\begin{equation}\label{37b}
s=1+\frac{2}{N}-\frac{1}{l}\,,
\end{equation}
\begin{equation}\label{l}
 \frac{N}{2}<l<\min\{m_1\,,m_2\},
 \end{equation}
  \begin{equation}\label{barC}
 \overline C_1=\left(\frac{s}{s-1}\right)^{s} \frac{1}{C_s^2}\left(\frac{2}{\bar k}\right)^{\frac{1}{l}-\frac{1}{m_1}}\,,
\quad
\overline C_2=\left(\frac{s}{s-1}\right)^{s} \frac{1}{C_s^2}\left(\frac{2}{\bar k}\right)^{\frac{1}{l}-\frac{1}{m_2}}\,,
 \end{equation}
 and
\begin{equation}\label{38}
\begin{aligned}
C_1=\overline C_1\,\|v\|_{L^1(B_R)}^{\frac{1}{l}-\frac{1}{m_1}}, \ \ \ C_2=\overline C_2\,\|v\|_{L^1(B_R)}^{\frac{1}{l}-\frac{1}{m_2}}\,.
\end{aligned}
\end{equation}
\end{proposition}

\begin{remark}\label{remark2}
\rm If in Proposition \ref{prop1} we further assume that there exists a constant $k_0>0$ such that
\[\max(\|v \|_{L^1(B_R)}, \|f_1 \|_{L^{m_1}(B_R)},\|f_2 \|_{L^{m_2}(B_R)}) \leq k_0 
\quad \textrm{ for all }\,\, R>0, \]
then from \eqref{eqa3}, we infer that the bound from above on $\|v\|_{L^{\infty}(B_R)}$ is independent of $R$. This fact will
have a key role in the proof of global existence for problem \eqref{problema}.
\end{remark}

\subsection{Proof of Proposition \ref{prop1}}
Let us first define
\begin{equation}\label{21}
G_k(v):=v-T_k(v) \,\,,
\end{equation}
\begin{equation*}
g(k):= \int_{B_R}|G_k(v)|\,d\mu.
\end{equation*}
For any $R>0$, for $v\in L^1(B_R)$, we set
\begin{equation}\label{23ab}
A_k:= \{x\in B_R:\,|v(x)|>k\}.
\end{equation}
We first state two technical Lemmas.
\begin{lemma}\label{lemma0}
Let $v\in L^1(B_R)$. Then $g(k)$ is differentiable almost everywhere in $(0,+\infty)$ and $$g'(k)=-\mu(A_{k}).$$
\end{lemma}

We omit the proof since it is identical to the one given in \cite{BC}.

\begin{lemma}\label{lemma1}
Let $v\in L^1(B_R)$. Let $\overline k>0$. Suppose that there exist $C>0$ and $s>1$ such that
\begin{equation}\label{23}
g(k)\le C\mu(A_k)^{s} \quad \text{for any}\,\,k\ge \bar k.
\end{equation}
Then $v\in L^{\infty}(B_R)$ and
\begin{equation}\label{eqa1}
\|v\|_{L^{\infty}(B_R)}\le C^{\frac{1}{s}}\frac{s}{s-1}\|v\|_{L^{1}(B_R)}^{1-\frac{1}{s}}+\bar k.
\end{equation}
\end{lemma}

\begin{remark}
\rm Observe that if $C$ in \eqref{23} does not depend on $R$ and, for some $k_0>0$,
\[\| v\|_{L^1(B_R)}\leq k_0 \quad \textrm{ for all }\,\, R>0,\]
then, in view of the estimate \eqref{eqa1},  the bound on $\|v\|_{L^{\infty}(B_R)}$ is independent of $R$.
\end{remark}

\begin{proof}[Proof of Lemma \ref{lemma1}]
Thanks to Lemma \ref{lemma0} together with hypotheses \eqref{23} we have that
$$
g'(k)=-\mu(A_{k})\le -\left [C^{-1}\,g(k)\right]^{\frac 1{s}},
$$
hence
$$
g(k) \le C\,[-g'(k)]^{s}.
$$
Integrating between $\bar k$ and $k$ we get
\begin{equation}\label{24b}
\int_{\bar k}^k\left(-\frac 1{C^{\frac 1 s}}\right )\, d\tau \ge \int_{\bar k}^k g'(\tau)\,g(\tau)^{-\frac{1}{s}} \, dg,
\end{equation}
that is:
$$
-C^{-\frac 1 s}( k-\bar k) \ge \frac{s}{s-1} \left [g(k)^{1-\frac 1 s} - g(\bar k)^{1-\frac 1 s}  \right].
$$
Using the definition of $g$, this can be rewritten as
$$
\begin{aligned}
g(k)^{1-\frac 1 s} &\le g\left(\bar k\right)^{1-\frac 1 s} - \frac{s-1}{s}\,C^{-\frac 1 s} (k-\bar k)\,\\
&\le \|v\|_{L^1(B_R)}^{1-\frac 1 s} - \frac{s-1}{s}\,C^{-\frac 1 s} (k-\bar k) \quad \text{for any}\,\, k>\bar k.
\end{aligned}
$$
Choose $$k=k_0=C^{\frac 1 s}\|v\|_{L^1(B_R)}^{1-\frac 1 s}\frac{s}{s-1}+\bar k,$$ and substitute it in the last inequality. Then
$g(k_0)\le0.$
Due to the definition of $g$ this is equivalent to
$$
\int_{B_R}|G_{k_0}(v)|\,d\mu =0 \,\,\,\iff \,\,\, |G_{k_0}(v)|=0 \,\,\,\iff \,\,\, |v|\le k_0.
$$
As a consequence we have
\begin{equation*}
\|v\|_{L^{\infty}(B_R)}\le k_0 = \frac{s}{s-1}C^{\frac 1 s}\|v\|_{L^1(B_R)}^{1-\frac 1 s}+\bar k.
\end{equation*}
\end{proof}

\begin{proof}[Proof of Proposition \ref{prop1}]
Take $G_k(v)$ as in \eqref{21} and $A_k$ as in \eqref{23ab}. From now one we write, with a slight abuse of notation, $$\|f\|_{L^q(B_R)}=\|f\|_{L^q}\,.$$ Since $G_k(v)\in H^1_0(B_R)$ and $G_k(v)\geq 0$, we can take $G_k(v)$ as test function in problem \eqref{25}. Then, by means of \eqref{S}, we get
\begin{equation}\label{32}
\begin{aligned}
\int_{B_R}\nabla u\cdot \nabla G_k(v)\, d\mu &\ge \int_{A_k}|\nabla v|^2\,d\mu \\
&\ge \int_{B_R}|\nabla G_k(v)|^2\,d\mu \\
&\ge C_s^2\left(\int_{B_R}| G_k(v)|^{2^*}\,d\mu\right )^{\frac{2}{2^*}}\,.
\end{aligned}
\end{equation}
If we now integrate on the right hand side of \eqref{25}, thanks to H\"older inequality, we get
\begin{equation}\label{33}
\begin{aligned}
\int_{B_R}(f_1+f_2)\,G_k(v)\,d\mu &= \int_{A_k}f_1\,G_k(v)\,d\mu + \int_{A_k}f_2\,G_k(v)\,d\mu \\
&\le\left(\int_{A_k}|G_k(v)|^{2^*}\,d\mu\right)^{\frac{1}{2^*}}\left [\left(\int_{A_k}|f_1|^{\frac{2N}{N+2}}\,d\mu\right)^{\frac{N+2}{2N}} + \left(\int_{A_k}|f_2|^{\frac{2N}{N+2}}\,d\mu\right)^{\frac{N+2}{2N}} \right]\\
&\le\left(\int_{B_R}|G_k(v)|^{2^*}\,d\mu\right)^{\frac{1}{2^*}}\left [\|f_1\|_{L^{m_1}}\mu(A_k)^{\frac{N+2}{2N}\left(1-\frac{2N}{m_1(N+2)}\right)}\right.\\
&\left.+ \|f_2\|_{L^{m_2}}\mu(A_k)^{\frac{N+2}{2N}\left(1-\frac{2N}{m_2(N+2)}\right)} \right]\,.
\end{aligned}
\end{equation}
Combining \eqref{32} and \eqref{33} we have
\begin{equation}\label{34}
\begin{aligned}
C_s^2\left(\int_{B_R}|G_k(v)|^{2^*}\,d\mu\right)^{\frac{1}{2^*}}
&\le\left [\|f_1\|_{L^{m_1}}\mu(A_k)^{\frac{N+2}{2N}\left(1-\frac{2N}{m_1(N+2)}\right)}\right.\\&\left.+ \|f_2\|_{L^{m_2}}\mu(A_k)^{\frac{N+2}{2N}\left(1-\frac{2N}{m_2(N+2)}\right)} \right]\,.
\end{aligned}
\end{equation}
Observe that
\begin{equation}\label{35}
\int_{B_R}|G_k(v)|\,d\mu\le \left(\int_{B_R}|G_k(v)|^{2^*}\,d\mu\right)^{\frac{1}{2^*}} \mu(A_k)^{\frac{N+2}{2N}}\,.
\end{equation}
We substitute \eqref{35} in \eqref{34} and we obtain
\begin{equation*}
\int_{B_R}|G_k(v)|\,d\mu\le\frac{1}{C_s^2}\left[\|f_1\|_{L^{m_1}}\mu(A_k)^{1+\frac{2}{N}-\frac{1}{m_1}}+\|f_2\|_{L^{m_2}}\mu(A_k)^{1+\frac{2}{N}-\frac{1}{m_2}}\right].
\end{equation*}
Using the definition of $l$ in \eqref{l}, for any $k\ge \overline k$, we can write
\begin{equation}\label{37}
\begin{aligned}
\int_{B_R}|G_k(v)|\,d\mu&\le\frac{1}{C_s^2}\,\mu(A_k)^{1+\frac{2}{N}-\frac{1}{l}}\left[\|f_1\|_{L^{m_1}}\mu(A_{\overline k})^{\frac{1}{l}-\frac{1}{m_1}}+\|f_2\|_{L^{m_2}}\mu(A_{\overline k})^{\frac{1}{l}-\frac{1}{m_2}}\right]
\end{aligned}
\end{equation}
Set
\begin{equation*}
C=\frac{1}{C_s^2}\left[\|f_1\|_{L^{m_1}}\left(\frac{2}{\bar k}\|v\|_{L^1(B_R)}\right)^{\frac{1}{l}-\frac{1}{m_1}}+\|f_2\|_{L^{m_2}}\left(\frac{2}{\bar k}\|v\|_{L^1(B_R)}\right)^{\frac{1}{l}-\frac{1}{m_2}}\right]\,.
\end{equation*}
Hence, by means of Chebychev inequality, \eqref{37} reads, for any $k\ge \bar k$,
\begin{equation}\label{37a}
\int_{B_R}|G_k(v)|\,d\mu \le C\,\mu(A_k)^{s}\,,
\end{equation}
where $s$ has been defined in \eqref{37b}. Now, \eqref{37a} corresponds to the hypotheses of Lemma \ref{lemma1}, hence the thesis of such lemma follows and we have
\begin{equation*}
\|v\|_{L^{\infty}}\le  \frac{s}{s-1}C^{\frac{1}{s}}\,\|v\|_{L^1}^{1-\frac{1}{s}}+\bar k\,.
\end{equation*}
Then the thesis follows thanks to \eqref{38}.
\end{proof}

\section{$L^q$ and smoothing estimates}\label{Lp}

\begin{lemma}\label{lemma3}
Let $1<p< m$. Let $M$ be such that inequality \eqref{P} holds. Suppose that $u_0\in L^{\infty}(B_R)$, $u_0\ge0$. Let $u$ be the solution of problem \eqref{problemapalla}. Then, for any $1<q<+\infty$, for some constant $C=C(q)>0$,
one has
\begin{equation}\label{47}
\|u(t)\|_{L^q(B_R)} \le e^{C(q) t}\|u_0\|_{L^q(B_R)}\quad \textrm{ for all }\,\, t>0\,.
\end{equation}
\end{lemma}

\begin{proof}
Let $x\in\R$, $x\ge0$, $1<p< m$, $\varepsilon >0$. Then, for any $1<q<+\infty$, due to Young's inequality, it follows that
\begin{equation}\label{48}
\begin{aligned}
x^{p+q-1}&=x^{(m+q-1)(\frac{p-1}{m-1})}x^{q(\frac{m-p}{m-1})}\\
&\le\varepsilon x^{(m+q-1)(\frac{p-1}{m-1})(\frac{m-1}{p-1})} + \left(\frac{1}{\varepsilon}\frac{p-1}{m-1}\right)^{\frac{p-1}{m-p}}x^{q(\frac{m-p}{m-1})(\frac{m-1}{m-p})}\\
&=\varepsilon x^{m+q-1}+\left(\frac{1}{\varepsilon}\frac{p-1}{m-1}\right)^{\frac{p-1}{m-p}}x^{q}.
\end{aligned}
\end{equation}
Since $u_0$ is bounded and $T_k(u^p)$ is a bounded and Lipschitz function, by standard results, there exists a unique solution of problem \eqref{problemapalla} in the sense of Definition \ref{4}; moreover, $u\in C\big([0, T]; L^q(B_R)\big)$. We now multiply both sides of the differential equation in problem \eqref{problemapalla} by $u^{q-1}$ and integrate by parts. This can be justified by standard tools, by an approximation procedure. Using the fact that
\[T_k(u^p)\leq u^p,\]
thanks to the Poincar\'{e} inequality, we obtain for
all $t>0$
$$
\frac{1}{q}\frac{d}{dt} \|u(t)\|_{L^q(B_R)}^q\le -\frac{4m(q-1)}{(m+q-1)^2}C_p^2 \|u(t)\|_{L^{m+q-1}(B_R)}^{m+q-1}+ \|u(t)\|_{L^{p+q-1}(B_R)}^{p+q-1}.
$$
Now, using inequality \eqref{48}, we obtain
$$
\frac{1}{q}\frac{d}{dt} \|u(t)\|_{L^q(B_R)}^q\le -\frac{4m(q-1)}{(m+q-1)^2}C_p^2 \|u(t)\|_{L^{m+q-1}(B_R)}^{m+q-1}+ \varepsilon \|u(t)\|_{L^{m+q-1}(B_R)}^{m+q-1} + C(\varepsilon)\|u(t)\|_{L^q(B_R)}^q,
$$
where $C(\varepsilon)=\left(\frac{1}{\varepsilon}\frac{p-1}{m-1}\right)^{\frac{p-1}{m-p}}.$ Thus, for every $\varepsilon>0$ so small that
$$0<\varepsilon<\frac{4m(q-1)}{(m+q-1)^2}C_p^2,$$ we have
$$
\frac{1}{q}\frac{d}{dt} \|u(t)\|_{L^q(B_R)}^q\le C(\varepsilon)\|u(t)\|_{L^q(B_R)}^q\,.
$$
Hence, we can find $C=C(q)>0$ such that
$$
\frac{d}{dt} \|u(t)\|_{L^q(B_R)}^q\le C(q)\|u(t)\|_{L^q(B_R)}^q \quad \textrm{ for all }\,\, t>0\,.
$$
If we set $y(t):=\|u(t)\|_{L^q(B_R)}^q$, the previous inequality reads
$$
y'(t)\le  C(q)y(t) \quad \textrm{ for all }\,\, t\in (0, T)\,.
$$
Thus the thesis follows.
\end{proof}

Note that for the constant $C(q)$ in Lemma \ref{lemma3} does not depend on $R$ and $k>0$; moreover, we have that
\[C(q)\to +\infty \quad \textrm{ as }\,\, q\to +\infty\,.\]

We shall use the following Aronson-Benilan type estimate (see \cite{AB}; see also \cite[Proposition 2.3]{Sacks}).
\begin{proposition}\label{prop2}
Let $1<p< m$, $u_0\in H_0^1(B_R) \cap L^{\infty}(B_R)$, $u_0\ge 0$. Let $u$ be the solution to problem \eqref{problemapalla}. Then, for a.e. $t\in(0,T)$,
\begin{equation*}
-\Delta u^m(\cdot,t) \le u^p(\cdot, t)+\frac{1}{(m-1)t} u(\cdot,t) \quad \text{in}\,\,\,\mathfrak{D}'(B_R).
\end{equation*}
\end{proposition}
\begin{proof}
By arguing as in \cite{AB}, \cite[Proposition 2.3]{Sacks} we get
\[
-\Delta u^m(\cdot,t) \le T_k[u^p(\cdot, t)]+\frac{1}{(m-1)t} u(\cdot,t) \leq u^p(\cdot, t)+\frac{1}{(m-1)t} u(\cdot,t) \quad \text{in}\,\,\,\mathfrak{D}'(B_R),
\]
since $T_k(u^p)\leq u^p\,.$
\end{proof}

\begin{proposition}\label{teo2}
Let $1<p<m$, $R>0, u_0\in  L^{\infty}(B_R)$, $u_0\ge 0$. Let $u$ be the solution to problem \eqref{problemapalla}. Let $M$ be such that inequality \eqref{S} holds.
Then there exists $\Gamma=\Gamma(p, m, N, C_s)>0$ such that, for all $t>0$,
\begin{equation}\label{eqa7}
\begin{aligned}
\|u(t)\|_{L^{\infty}(B_R)}&\le \Gamma \left\{ \left [e^{Ct}\|u_{0}\|_{L^m(B_R)}\right ]^{\frac{2m}{2m +N(m-p)}}\right.\\&\left.+\left [e^{Ct}\|u_{0}\|_{L^m(B_R)}\right ]^{\frac{2m}{2m+N(m-1)}} \left [\frac{1}{(m-1)t}\right]^{\frac{N}{2m+N(m-1)}}\right\}\,;
\end{aligned}
\end{equation}
here the constant $C=C(m)>0$ is the one given in Lemma \ref{lemma3}\,.
\end{proposition}

\begin{remark}\label{remark3}
\rm If in Proposition \ref{teo2}, in addition, we assume that for some $k_0>0$
\[\|u_0\|_{L^m(B_R)}\leq k_0\quad \textrm{ for every }\,\, R>0\,,\]
then the bound from above for $\|u(t)\|_{L^{\infty}(B_R)}$ in \eqref{eqa7} is independent of $R$.
\end{remark}

\begin{proof}[Proof of Proposition \ref{teo2}]
Let us set $w=u(\cdot,t)$. Observe that $w^m\in H_0^1(B_R)$ and $w\ge0$. Due to Proposition \ref{prop2} we know that
\begin{equation}\label{50a}
-\Delta(w^m) \le \left [w^p+\frac{w}{(m-1)t} \right].
\end{equation}
Observe that, since $u_0\in L^{\infty}(B_R)$ also $w\in L^{\infty}(B_R)$. Let $q\ge1$ and
$$r_1>\max\left\{\frac{q}{p}, \frac N2\right\}, \quad r_2>\max\left\{q, \frac N2\right\}\,.$$
We can apply Proposition \ref{prop1}
with
\[r_1=m_1, \quad r_2=m_2, \quad
 \frac{N}{2}<l<\min\{m_1\,,m_2\}\,.
\]
So, we have that
\begin{equation}\label{50}
\|w\|_{L^{\infty}(B_R)}^m\le \left\{C_1(r_1)\|w^p\|_{L^{r_1}(B_R)}+\gamma C_2(r_2)\|w\|_{L^{r_2}(B_R)} \right\}^{\frac{1}{s}} \|w\|_{L^{m}(B_R)}^{m\frac{s-1}{s}}+\bar k\,,
\end{equation}
where
$
s=1+ 2 /N-1/l\,
$
and $ \gamma=1/[(m-1)t]$. Thanks to H\"{o}lder inequality and Young's inequality with exponents
$$
\alpha_1=\frac{sm}{p-\frac{q}{r_1}}\,>1,\quad\quad \beta_1=\frac{sm}{sm-\left(p-\frac{q}{r_1}\right)}>1.
$$
we obtain, for any $\varepsilon_1>0$
\begin{equation}\label{eq2}
\begin{aligned}
\|w^p\|_{L^{r_1}(B_R)} &= \left\|w^{p-q/r_1+q/r_1}\right\|_{L^{r_1}(B_R)}
=\left[ \int_{B_R}w^{r_1(p-q/r_1)} w^{q} d\mu\right]^{\frac{1}{r_1}}\\
&\le\left[ \|w^{r_1\left(p-q/r_1\right)}\|_{L^{\infty}(B_R)} \|w^{q}\|_{L^1(B_R)} \right]^{\frac{1}{r_1}}\\
&=\|w\|_{L^{\infty}(B_R)}^{p-q/r_1}\left(\int_{B_R}w^{q}\,d\mu\right)^{\frac{1}{r_1}}=\left\|w\right\|_{L^{\infty}(B_R)}^{p-q/r_1}\left \|w\right\|_{L^{q}(B_R)}^{q/r_1} \\
&\le\frac{\varepsilon_1^{\alpha_1}}{\alpha_1}\left\|w\right\|_{L^{\infty}(B_R)}^{\frac{sm}{p-q/r_1}(p-q/r_1)} + \frac{\alpha_1-1}{\alpha_1}\varepsilon_1^{-\frac{\alpha_1}{\alpha_1-1}}\left\|w\right\|_{L^{q}(B_R)}^{\frac{\beta_1q}{r_1}} .
\end{aligned}
\end{equation}
We set $$\delta_1:=\frac{\varepsilon_1^{\alpha_1}}{\alpha_1},\ \ \
\eta(x)=\dfrac{x-1}{x^{\frac{x}{x-1}}}.
$$
Thus from \eqref{eq2} we obtain
\begin{equation}\label{49}
\left\|w^p\right\|_{L^{r_1}(B_R)} \le\delta_1\left\|w\right\|_{L^{\infty}(B_R)}^{sm} + \frac{\eta(\alpha_1)}{\delta_1^{\frac{1}{\alpha_1-1}}}\left\|w\right\|_{L^{q}(B_R)}^{\frac{smq}{r_1}\frac{1}{sm-(p-q/r_1)}}.
\end{equation}
Similarly, again thanks to H\"{o}lder inequality and Young's inequality with exponents
$$
\alpha_2=\frac{sm}{1-\frac{q}{r_2}}\,>1,\quad\quad \beta_2=\frac{sm}{sm-\left(1-\frac{q}{r_2}\right)}>1.
$$
we obtain, for any $\varepsilon_2>0$
\begin{equation*}
\begin{aligned}
\left\|w\right\|_{L^{r_2}(B_R)} &\le \left\|w^{1-q/r_2+q/r_2}\right\|_{L^{r_2}(B_R)}\le\left\|w\right\|_{L^{\infty}(B_R)}^{1-q/r_2}\left\|w\right\|_{L^{q}(B_R)}^{q/r_2} \\
&\le\frac{\varepsilon_2^{\alpha_2}}{\alpha_2}\left\|w\right\|_{L^{\infty}(B_R)}^{\frac{sm}{1-q/r_2}\left(1-\frac{q}{r_2}\right)} + \frac{\alpha_2-1}{\alpha_2}\varepsilon_2^{-\frac{\alpha_2}{\alpha_2-1}}\left\|w\right\|_{L^{q}(B_R)}^{\frac{\beta_2q}{r_2}} .
\end{aligned}
\end{equation*}
We set $\delta_2:=\frac{\varepsilon_2^{\alpha_2}}{\alpha_2}$ and thus we obtain
\begin{equation}\label{49a}
\left\|w\right\|_{L^{r_2}(B_R)} \le\delta_2\left\|w\right\|_{L^{\infty}(B_R)}^{sm} + \frac{\eta(\alpha_2)}{\delta_2^{\frac{1}{\alpha_2-1}}}\left\|w\right\|_{L^{q}(B_R)}^{\frac{smq}{r_2}\frac{1}{sm-(1-q/r_2)}}.
\end{equation}
Plugging \eqref{49} and \eqref{49a} into \eqref{50} we obtain
$$
\begin{aligned}
\|w\|_{L^{\infty}(B_R)}^{ms} &\le 2^{s-1}\left\{\left[C_1\|w^p\|_{L^{r_1}(B_R)}+\gamma C_2\|w\|_{L^{r_2}(B_R)} \right] \|w\|_{L^{m}(B_R)}^{m(s-1)} +\bar k^{s}\right\}\\
&\le2^{s-1}\left\{C_1\left[\delta_1\left\|w\right\|_{L^{\infty}(B_R)}^{sm} + \frac{\eta(\alpha_1)}{\delta_1^{\frac{1}{\alpha_1-1}}}\left\|w\right\|_{L^{q}(B_R)}^{\frac{smq}{r_1}\frac{1}{sm-(p-q/r_1)}} \right] \right .\\ &\left.+ \gamma C_2 \left[ \delta_2\left\|w\right\|_{L^{\infty}(B_R)}^{sm} + \frac{\eta(\alpha_2)}{\delta_2^{\frac{1}{\alpha_2-1}}}\left\|w\right\|_{L^{q}(B_R)}^{\frac{smq}{r_2}\frac{1}{sm-(1-q/r_2)}} \right] \right\} \|w\|_{L^{m}(B_R)}^{m(s-1)}+2^{s-1}\bar k^{s}.
\end{aligned}
$$
Without loss of generality we can assume that $\|w\|_{L^{m}(B_R)}^m\neq 0$. Choosing $\varepsilon_1, \varepsilon_2$ such that
$$
\delta_1=\dfrac{1}{4C_1\|w\|_{L^{m}(B_R)}^{m(s-1)}2^{s-1} }
\ \ \ \delta_2=\dfrac{1}{4\gamma\,C_2\|w\|_{L^{m}(B_R)}^{m(s-1)}2^{s-1}}
$$
we thus have
$$
\begin{aligned}
\frac 1 2 \|w\|_{L^{\infty}(B_R)}^{sm} &\le 4^{\frac{1}{\alpha_1-1}}\eta(\alpha_1)\left(2^{s-1}C_1\|w\|_{L^{m}(B_R)}^{m(s-1)}\right)^{\frac{\alpha_1}{\alpha_1-1}}\left\|w\right\|_{L^{q}(B_R)}^{\frac{smq}{r_1}\frac{1}{sm-(p-q/r_1)}} \\
& + 4^{\frac{1}{\alpha_2-1}}\eta(\alpha_2)\left(2^{s-1}\gamma C_2 \|w\|_{L^{m}(B_R)}^{m(s-1)}\right)^{\frac{\alpha_2}{\alpha_2-1}}\left\|w\right\|_{L^q(B_R)}^{\frac{smq}{r_2}\frac{1}{sm-(1-q/r_2)}}\\
&+2^{s-1}\bar k^{s}.
\end{aligned}
$$
This reduces to
$$
\begin{aligned}
\|w\|_{L^{\infty}(B_R)} &\le (2)^{\frac {1}{sm}}\,4^{\frac{1}{sm(\alpha_1-1)}}\eta(\alpha_1)^{\frac{1}{sm}}\left(2^{\frac{s-1}{sm}}C_1^{\frac{1}{sm}}\|w\|_{L^{m}(B_R)}^{\frac{s-1}{s}}\right)^{\frac{\alpha_1}{\alpha_1-1}}\left\|w\right\|_{L^{q}(B_R)}^{\frac{q}{r_1}\frac{1}{sm-(p-q/r_1)}} \\
&+ (2)^{\frac {1}{sm}}\,4^{\frac{1}{sm(\alpha_2-1)}}\eta(\alpha_2)^{\frac{1}{sm}}\left(2^{\frac{s-1}{sm}}\gamma^{\frac{1}{sm}}C_2^{\frac{1}{sm}}\|w\|_{L^{m}(B_R)}^{\frac{s-1}{s}}\right)^{\frac{\alpha_2}{\alpha_2-1}}\left\|w\right\|_{L^{q}(B_R)}^{\frac{q}{r_2}\frac{1}{sm-(1-q/r_2)}}\\
&+(2)^{\frac {1}{sm}}\left(2^{\frac{s-1}{s}}\bar k\right)^{\frac 1m}.
\end{aligned}
$$
This can be rewritten as
\begin{equation}\label{eq3}
\begin{aligned}
\|w\|_{L^{\infty}(B_R)} &\le \left[\eta(\alpha_1)\left(2^{\alpha_1s+1}C_1^{\alpha_1}\right)^{\frac{1}{\alpha_1-1}} \right]^{\frac {1}{sm}}\|w\|_{L^{m}(B_R)}^{\frac{\alpha_1}{\alpha_1-1}\frac{s-1}{s}}\left\|w\right\|_{L^{q}(B_R)}^{\frac{q}{r_1}\frac{1}{sm-(p-q/r_1)}} \\
&+ \left[\eta(\alpha_2)\left(2^{\alpha_2s+1}\gamma^{\alpha_2}C_2^{\alpha_2}\right)^{\frac{1}{\alpha_2-1}} \right]^{\frac {1}{sm}}\|w\|_{L^{m}(B_R)}^{\frac{\alpha_2}{\alpha_2-1}\frac{s-1}{s}}\left\|w\right\|_{L^{q}(B_R)}^{\frac{q}{r_2}\frac{1}{sm-(1-q/r_2)}}\\
&+\left(2\bar k\right)^{\frac{1}{m}}.
\end{aligned}
\end{equation}
Now we use the definitions of $C_1$, $C_2, \overline C_1, \overline C_2$ introduced in \eqref{38} and \eqref{barC}, obtaining
\begin{equation*}
\begin{aligned}
\|w\|_{L^{\infty}(B_R)} &\le \left[\eta(\alpha_1)\left(2^{\alpha_1s+1}\,\overline C_1^{\alpha_1}\right)^{\frac{1}{\alpha_1-1}} \right]^{\frac {1}{sm}}\|w\|_{L^{m}(B_R)}^{\frac{\alpha_1}{\alpha_1-1}\frac 1s\left(s-1+\frac 1l-\frac {1}{r_1}\right)}\left\|w\right\|_{L^{q}(B_R)}^{\frac{q}{r_1}\frac{1}{sm-(p-q/r_1)}} \\
&+ \left[\eta(\alpha_2)\left(2^{\alpha_2s+1}\gamma^{\alpha_2}\overline C_2^{\alpha_2}\right)^{\frac{1}{\alpha_2-1}} \right]^{\frac {1}{sm}}\|w\|_{L^{m}(B_R)}^{\frac{\alpha_2}{\alpha_2-1}\frac 1s\left(s-1+\frac 1l-\frac {1}{r_2}\right)}\left\|w\right\|_{L^{q}(B_R)}^{\frac{q}{r_2}\frac{1}{sm-(1-q/r_2)}}\\
&+\left(2\bar k\right)^{\frac{1}{m}}.
\end{aligned}
\end{equation*}
By taking limits as $r_1\longrightarrow +\infty$ and $r_2\longrightarrow +\infty$ we have
$$
\begin{aligned}
&\frac{\alpha_1}{\alpha_1-1}\longrightarrow \frac{ms}{ms-p};\\
&\frac{\alpha_2}{\alpha_2-1}\longrightarrow \frac{ms}{ms-1};\\
&\eta(\alpha_1)\longrightarrow\left[\frac{p}{ms}\right]^{\frac{p}{ms-p}}\left\{1-\frac{p}{ms}\right\};\\
&\eta(\alpha_2)\longrightarrow\left[\frac{1}{ms}\right]^{\frac{1}{ms-1}}\left\{1-\frac{1}{ms}\right\};\\
&\overline C_1\longrightarrow\left(\frac{s}{s-1}\right)^s\frac{1}{C_s^2}\left(\frac{2}{\overline k}\right)^{\frac 1l};\\
&\overline C_2\longrightarrow\left(\frac{s}{s-1}\right)^s\frac{1}{C_s^2}\left(\frac{2}{\overline k}\right)^{\frac 1l}.
\end{aligned}
$$
Moreover we define
$$
\begin{aligned}
&\tilde \Gamma_1:=\left[\left(\frac{p}{ms}\right)^{\frac{p}{ms-p}}\left(\frac{ms-p}{ms}\right)\right]^{\frac{1}{ms}}\left[2^{ms^2-p}\left(\frac{s}{s-1}\right)^s\frac{1}{C_s^2}\left(\frac{2}{\overline k}\right)^{\frac 1l}\right]^{\frac{1}{ms-p}}, \\
&\tilde \Gamma_2:=\left[\left(\frac{1}{ms}\right)^{\frac{1}{ms-1}}\left(\frac{ms-1}{ms}\right)\right]^{\frac{1}{ms}}\left[2^{ms^2-1}\left(\frac{s}{s-1}\right)^s\frac{1}{C_s^2}\left(\frac{2}{\overline k}\right)^{\frac 1l}\right]^{\frac{1}{ms-1}}, \\
&\tilde \Gamma:=\max\{\tilde\Gamma_1\,\,,\tilde\Gamma_2\}.
\end{aligned}
$$
Hence by \eqref{eq3} we get
\begin{equation}\label{51}
\|w\|_{L^{\infty}(B_R)} \le \tilde \Gamma \left [ \|w\|_{L^{m}(B_R)}^{\frac{m}{ms-p}\left(s-1+\frac{1}{l}\right)} +  \|w\|_{L^{m}(B_R)}^{\frac{m}{ms-1}\left(s-1+\frac{1}{l}\right)} \gamma^{\frac{1}{ms-1}}\right ]+\left(2\bar k\right)^{\frac{1}{m}}.
\end{equation}
Letting $l\to +\infty$ in \eqref{51}, we can infer that
\begin{equation}\label{eqa5}
\|w\|_{L^{\infty}(B_R)} \le  \Gamma \left [ \|w\|_{L^{m}(B_R)}^{\frac{2m}{2m +N(m-p)}} +  \|w\|_{L^{m}(B_R)}^{\frac{2m}{2m+N(m-1)}} \gamma^{\frac{N}{2m+N(m-1)}}\right ]+\left(2\bar k\right)^{\frac{1}{m}},
\end{equation}
where
\begin{equation*}
\begin{aligned}
&\Gamma_1=\left(1-\frac{pN}{m(N+2)}\right)^{\frac{N}{m(N+2)}}\,2\,\left[\left(\frac{pN}{m(N+2)}\right)^{\frac{pN}{m(N+2)}}2^{2m\left(1+\frac 2N\right)}\left(\frac{N+2}{N}\right)^{\frac{N+2}{N}}\frac{1}{C_s^2}\right]^{\frac{N}{2m+N(m-p)}}, \\
&\Gamma_2=\left(1-\frac{N}{m(N+2)}\right)^{\frac{N}{m(N+2)}}\,2\,\left[\left(\frac{N}{m(N+2)}\right)^{\frac{N}{m(N+2)}}2^{2m\left(1+\frac 2N\right)}\left(\frac{N+2}{N}\right)^{\frac{N+2}{N}}\frac{1}{C_s^2}\right]^{\frac{N}{2m+N(m-1)}};\\
&\Gamma=\max\{\Gamma_1\,;\,\Gamma_2\}.
\end{aligned}
\end{equation*}
Letting $\bar k\to 0$ in \eqref{eqa5} we obtain
\begin{equation}\label{eqa6}
\|w\|_{L^{\infty}(B_R)} \le  \Gamma \left [ \|w\|_{L^{m}(B_R)}^{\frac{2m}{2m +N(m-p)}} +  \|w\|_{L^{m}(B_R)}^{\frac{2m}{2m+N(m-1)}} \gamma^{\frac{N}{2m+N(m-1)}}\right ]\,.
\end{equation}
Finally, since $u_0\in L^{\infty}(B_R)$, we can apply Lemma \ref{lemma3} to $w$ with $q=m$. Thus from \eqref{47} with $q=m$ and \eqref{eqa6}, the thesis follows.
\end{proof}

\section{Proof of Theorems \ref{teoesistenza}, \ref{blowup}}\label{proofs}

\begin{proof}[Proof of Theorem \ref{teoesistenza}]
Let $\{u_{0,h}\}_{h\ge 0}$ be a sequence of functions such that
\begin{equation*}
\begin{aligned}
&u_{0,h}\in L^{\infty}(M)\cap C_c^{\infty}(M) \,\,\,\text{for all} \,\,h\ge 0, \\
&u_{0,h}\ge 0 \,\,\,\text{for all} \,\,h\ge 0, \\
&u_{0, h_1}\leq u_{0, h_2}\,\,\,\text{for any } h_1<h_2,  \\
&u_{0,h}\longrightarrow u_0 \,\,\, \text{in}\,\, L^m(M)\quad \textrm{ as }\, h\to +\infty\,.
\end{aligned}
\end{equation*}
For any $R>0, k>0, h>0,$ consider the problem
\begin{equation}\label{5}
\begin{cases}
u_t= \Delta u^m +T_k(u^p) &\text{in}\,\, B_R\times (0,+\infty)\\
u=0& \text{in}\,\, \partial B_R\times (0,\infty)\\
u=u_{0,h} &\text{in}\,\, B_R\times \{0\}\,. \\
\end{cases}
\end{equation}
From standard results it follows that problem \eqref{5} has a solution $u_{h,k}^R$ in the sense of Definition \ref{4}; moreover, $u^R_{h,k}\in C\big([0, T]; L^q(B_R)\big)$ for any $q>1$. Hence, it satisfies the inequalities in Lemma \ref{lemma3} and in Proposition \ref{teo2}, i.e., for any $t\in(0,+\infty)$,
\begin{equation}\label{6}
\|u_{h,k}^R(t)\|_{L^m(B_R)}\,\le\, e^{ C t}\|u_{0,h}\|_{L^m(B_R)};
\end{equation}

\begin{equation}\label{7}
\begin{aligned}
\|u_{h, k}^R(t)\|_{L^{\infty}(B_R)}&\le \Gamma \left\{ \left [e^{Ct}\|u_{0,h}\|_{L^m(B_R)}\right ]^{\frac{2m}{2m +N(m-p)}}\right.\\&\left.+\left [e^{Ct}\|u_{0,h}\|_{L^m(B_R)}\right ]^{\frac{2m}{2m+N(m-1)}} \left [\frac{1}{(m-1)t}\right]^{\frac{N}{2m+N(m-1)}}\right\}\,.
\end{aligned}
\end{equation}
In addition, for any $\tau\in (0, T), \zeta\in C^1_c((\tau, T)), \zeta\geq 0$, $\max_{[\tau, T]}\zeta'>0$,
\begin{equation}\label{eqcont1}
\begin{aligned}
\int_{\tau}^T \zeta(t) \left[\big((u^R_{h, k})^{\frac{m+1}2}\big)_t\right]^2 d\mu dt &\leq \max_{[\tau, T]}\zeta' \bar C \int_{B_R}(u_{h, k}^R)^{m+1}(x, \tau)d\mu\\
&+ \bar C \max_{[\tau, T]}\zeta \int_{B_R} F\big(u^{R}_{h, k}(x,T)\big)d\mu\\
&\leq  \max_{[\tau, T]}\zeta'(t)\bar C \|u^R_{h, k}(\tau)\|_{L^\infty(B_R)}\|u^R_{h, k}(\tau)\|_{L^m(B_R)}^m \\
&+\frac {\bar C}{m+p}\|u^R_{h, k}(T)\|^p_{L^\infty(B_R)}\|u^R_{h, k}(T)\|_{L^m(B_R)}^m
\end{aligned}
\end{equation}
where
\[F(u)=\int_0^u s^{m-1+p} \, ds\,,\]
and $\bar C>0$ is a constant only depending on $m$. Inequality \eqref{eqcont1} is formally obtained by multiplying the differential inequality in problem \eqref{problemapalla} by $\zeta(t)[(u^m)_t]$, and integrating by parts; indeed, a standard approximation procedure is needed (see \cite[Lemma 3.3]{GMPo} and \cite[Theorem 13]{ACP}).

\smallskip

Moreover, as a consequence of Definition \ref{4}, for any $\varphi \in C_c^{\infty}(B_R\times[0,T])$ such that $\varphi(x,T)=0$ for any $x\in B_R$, $u_{h,k}^R$ satisfies
\begin{equation}\label{8}
\begin{aligned}
-\int_0^T\int_{B_R}u_{h,k}^R\,\varphi_t\,d\mu\,dt =&\int_0^T\int_{B_R} (u_{h,k}^R)^m\,\Delta\varphi\,d\mu\,dt\,+ \int_0^T\int_{B_R} T_k[(u_{h,k}^R)^p]\,\varphi\,d\mu\,dt \\
& +\int_{B_R} u_{0,h}(x)\,\varphi(x,0)\,d\mu.
\end{aligned}
\end{equation}
Observe that all the integrals in \eqref{8} are finite. Indeed, due to \eqref{6}, $u_{h,k}^R \in L^m(B_R\times(0,T))$ hence, since $p<m$, $u_{h,k}^R \in L^p(B_R\times(0,T))$ and $u_{h,k}^R \in L^1(B_R\times(0,T))$. Moreover, observe that, for any $h>0$ and $R>0$ the sequence of solutions $\{u_{h,k}^R\}_{k\ge0}$ is monotone increasing in $k$ hence it has a pointwise limit for $k\to \infty$. Let $u_h^R$ be such limit so that we have
$$
u_{h,k}^R\longrightarrow u_{h}^R \quad \text{as} \,\,\, k\to\infty \,\,\text{pointwise}.
$$
In view of \eqref{6}, \eqref{7}, the right hand side of \eqref{eqcont1} is independent of $k$. So, $(u^R_h)^{\frac{m+1}2}\in H^1((\tau, T); L^2(B_R))$. Therefore, $(u^R_h)^{\frac{m+1}2}\in C\big([\tau, T]; L^2(B_R)\big)$. We can now pass to the limit as $k\to +\infty$ in inequalities \eqref{6} and \eqref{7} arguing as follows. From inequality \eqref{6}, thanks to the Fatou's Lemma, one has for all $t>0$
\begin{equation}\label{10}
\begin{aligned}
\|u_{h}^R(t)\|_{L^m(B_R)}\leq  e^{ C t}\|u_{0,h}\|_{L^m(B_R)}.
\end{aligned}
\end{equation}
On the other hand, from \eqref{7}, since $u_{h,k}^R\longrightarrow u_{h}^R$ as $k\to \infty$ pointwise and the right hand side of \eqref{7} is independent of $k$, one has for all $t>0$
\begin{equation}\label{11}
\begin{aligned}
\|u_{h}^R(t)\|_{L^{\infty}(B_R)}&\le \le \Gamma \left\{ \left [e^{Ct}\|u_{0,h}\|_{L^m(B_R)}\right ]^{\frac{2m}{2m +N(m-p)}}\right.\\&\left.+\left [e^{Ct}\|u_{0,h}\|_{L^m(B_R)}\right ]^{\frac{2m}{2m+N(m-1)}} \left [\frac{1}{(m-1)t}\right]^{\frac{N}{2m+N(m-1)}}\right\}\,.
\end{aligned}
\end{equation}
Note that both \eqref{10} and \eqref{11} hold {\em for all} $t>0$, in view of the continuity property of $u$ deduced above.
Moreover, thanks to Beppo Levi's monotone convergence Theorem, it is possible to compute the limit as $k\to +\infty$ in the integrals of equality \eqref{8} and hence obtain that, for any $\varphi \in C_c^{\infty}(B_R\times(0,T))$ such that $\varphi(x,T)=0$ for any $x\in B_R$, the function $u_h^R$ satisfies
\begin{equation}\label{9}
\begin{aligned}
-\int_0^T\int_{B_R} u_{h}^R\,\varphi_t\,d\mu\,dt =&\int_0^T\int_{B_R} (u_{h}^R)^m\,\Delta\varphi\,d\mu\,dt+ \int_0^T\int_{B_R} (u_{h}^R)^p\,\varphi\,d\mu\,dt \\
& +\int_{B_R} u_{0,h}(x)\,\varphi(x,0)\,d\mu.
\end{aligned}
\end{equation}
Observe that, due to inequality \eqref{10}, all the integrals in \eqref{9} are finite, hence $u_h^R$ is a solution to problem \eqref{5}, where we replace $T_k(u^p)$ with $u^p$ itself, in the sense of Definition \ref{4}.

Let us now observe that, for any $h>0$, the sequence of solutions $\{u_h^R\}_{R>0}$ is monotone increasing in $R$, hence it has a pointwise limit as $R\to+\infty$. We call its limit function $u_h$ so that
$$
u_{h}^R\longrightarrow u_{h} \quad \text{as} \,\,\, R\to+\infty \,\,\text{pointwise}.
$$
In view of \eqref{6}, \eqref{7}, \eqref{10}, \eqref{11}, the right hand side of \eqref{eqcont1} is independent of $k$ and $R$. So, $(u_h)^{\frac{m+1}2}\in H^1((\tau, T); L^2(M))$. Therefore, $(u_h)^{\frac{m+1}2}\in C\big([\tau, T]; L^2(M)\big)$. Since $u_0\in L^m(M)$, there exists $k_0>0$ such that
\begin{equation}\label{eqag1}
\|u_{0h}\|_{L^m(B_R)}\leq k_0 \quad \forall\,\, h>0, R>0\,.
\end{equation}
Note that, in view of \eqref{eqag1}, the norms in  \eqref{10} and \eqref{11} do not depend on $R$ (see Proposition \ref{teo2}, Lemma \ref{lemma3} and Remark \ref{remark3}). Therefore, we pass to the limit as $R\to+\infty$ in \eqref{10} and \eqref{11}. By Fatou's Lemma,
\begin{equation}\label{12}
\begin{aligned}
\|u_{h}(t)\|_{L^m(M)}\leq  e^{Ct}\|u_{0,h}\|_{L^m(M)};
\end{aligned}
\end{equation}
furthermore, since $u_{h}^R\longrightarrow u_{h} $ as $R\to +\infty$ pointwise,
\begin{equation}\label{13}
\begin{aligned}
\|u_{h}(t)\|_{L^{\infty}(M)}&\le \Gamma \left\{ \left [e^{Ct}\|u_{0,h}\|_{L^m(M)}\right ]^{\frac{2m}{2m +N(m-p)}}\right.\\&\left.+\left [e^{Ct}\|u_{0,h}\|_{L^m(M)}\right ]^{\frac{2m}{2m+N(m-1)}} \left [\frac{1}{(m-1)t}\right]^{\frac{N}{2m+N(m-1)}}\right\}\,.
\end{aligned}
\end{equation}
Note that both \eqref{12} and \eqref{13} hold {\em for all} $t>0$, in view of the continuity property of $u^R_h$ deduced above.

Moreover, again by monotone convergence, it is possible to compute the limit as $R\to +\infty$ in the integrals of equality \eqref{9} and hence obtain that, for any $\varphi \in C_c^{\infty}(M\times(0,T))$ such that $\varphi(x,T)=0$ for any $x\in M$, the function $u_h$ satisfies,
\begin{equation}\label{14}
\begin{aligned}
-\int_0^T\int_{M} u_{h}\,\varphi_t\,d\mu\,dt =&\int_0^T\int_{M} (u_{h})^m\,\Delta\varphi\,d\mu\,dt+ \int_0^T\int_{M} (u_{h})^p\,\varphi\,d\mu\,dt \\
& +\int_{M} u_{0,h}(x)\,\varphi(x,0)\,d\mu.
\end{aligned}
\end{equation}
Observe that, due to inequality \eqref{12}, all the integrals in \eqref{14} are well posed hence $u_h$ is a solution to problem \eqref{problema}, where we replace $u_0$ with $u_{0,h}$, in the sense of Definition \ref{1}.
Finally, let us observe that $\{u_{0,h}\}_{h\ge0}$ has been chosen in such a way that
$$
u_{0,h}\longrightarrow u_0 \,\,\, \text{in}\,\, L^m(M)
$$
Observe also that $\{u_{h}\}_{h\ge0}$ is a monotone increasing function in $h$ hence it has a limit as $h\to+\infty$. We call $u$ the limit function.
In view  \eqref{6}, \eqref{7}, \eqref{10}, \eqref{11}, \eqref{12}, \eqref{13},  the right hand side of \eqref{eqcont1} is independent of $k, R$ and $h$. So, $u^{\frac{m+1}2}\in H^1((\tau, T); L^2(M))$. Therefore, $u^{\frac{m+1}2}\in C\big([\tau, T]; L^2(M)\big)$. Hence, we can pass to the limit as $h\to +\infty$ in \eqref{12} and \eqref{13} and similarly to what we have seen above, we get
\begin{equation}\label{15}
\|u(t)\|_{L^m(M)}\le  e^{Ct}\|u_{0}\|_{L^m(M)},
\end{equation}
and
\begin{equation}\label{16}
\begin{aligned}
\|u(t)\|_{L^{\infty}(M)}\,
&\le \Gamma \left\{ \left [e^{Ct}\|u_{0}\|_{L^m(M)}\right ]^{\frac{2m}{2m +N(m-p)}}\right.\\&\left.+\left [e^{Ct}\|u_{0}\|_{L^m(M)}\right ]^{\frac{2m}{2m+N(m-1)}} \left [\frac{1}{(m-1)t}\right]^{\frac{N}{2m+N(m-1)}}\right\}\,.
\end{aligned}
\end{equation}
Note that both \eqref{15} and \eqref{16} hold {\em for all} $t>0$, in view of the continuity property of $u$ deduced above.

Moreover, again by monotone convergence, it is possible to compute the limit as $h\to+\infty$ in the integrals of equality \eqref{14} and hence obtain that, for any $\varphi \in C_c^{\infty}(M\times(0,T))$ such that $\varphi(x,T)=0$ for any $x\in M$, the function $u$ satisfies,
\begin{equation}\label{17}
\begin{aligned}
-\int_0^T\int_{M} u\,\varphi_t\,d\mu\,dt =&\int_0^T\int_{M} u^m\,\Delta\varphi\,d\mu\,dt+ \int_0^T\int_{M} u^p\,\varphi\,d\mu\,dt \\
& +\int_{M} u_{0}(x)\,\varphi(x,0)\,d\mu.
\end{aligned}
\end{equation}
Observe that, due to inequality \eqref{15}, all the integrals in \eqref{17} are finite, hence $u$ is a solution to problem \eqref{problema} in the sense of Definition \ref{1}.

\smallskip

\medskip

Finally, let us discuss \eqref{Lq}. Let $q>1$. If $u_0\in L^q(M)\cap L^m(M)$, we choose the sequence $u_{0h}$ so that it further satisfies
\[u_{0h}\to u_0 \quad \textrm{ in }\,\, L^q(M)\,\quad \textrm{ as }\, h\to +\infty\,.\]
We have that
\begin{equation}\label{6a}
\|u_{h,k}^R(t)\|_{L^q(B_R)}\,\le\, e^{ C t}\|u_{0,h}\|_{L^q(B_R)}.
\end{equation}
Hence, due to \eqref{6a}, letting $k\to +\infty, R\to +\infty, h\to +\infty$, by Fatou's Lemma we deduce \eqref{Lq}.
\end{proof}

\begin{proof}[Proof of Theorem \ref{blowup}]

We note in first place that the geometrical assumptions on $M$, in particular the upper curvature bound sec$\,\le -k_2<0$, ensure that inequalities \eqref{P} and \eqref{S} both hold on $M$, see e.g. \cite{Grig, Grig3}. Hence, all the result of Theorem \ref{teoesistenza} hold, in particular solutions corresponding to data $u_0\in{\textrm L}^m(M)$ exist globally in time.

Besides, it has been shown in \cite{GMPv} that if $u_0$ is a continuous, nonnegative, nontrivial datum, which is sufficiently large in the sense given in the statement, under the lower curvature bound being assumed here the corresponding solution $u$ satisfies the bound
\[
u(x,t)\ge C \zeta(t) \left[ 1- \frac r a \eta(t) \right]_+^{\frac 1{m-1}}\,\qquad \forall t\in(0,S),\ \forall x\in M,
\]
possibly up to a finite time explotion time $S$, which has however been proved in the present paper not to exist. Here, the functions $\eta, \zeta$ are given by:
\[\zeta(t):=(\tau+t)^{\alpha}\,, \quad \eta(t):=(\tau +t)^{-\beta} \quad \text{for every } t\in [0, \infty) \, ,\]
where $C, \tau, R_0,\inf_{B_{R_0}}u_0$ must be large enough and one can take $0<\alpha<\frac 1{m-1} \, , \beta=\frac{\alpha(m-1)+1}2$. Clearly, $u$ then satisfies $\lim_{t\to+\infty}u(x,t)=+\infty$ for all $x\in M$, and hence $u$ enjoys the same property by comparison.
\end{proof}

\section{Proof of Theorems \ref{teoesistenza2}, \ref{teoesistenza3}}\label{weights}
For any $R>0$ we consider the following approximate problem
\begin{equation}\label{problemapalla2}
\begin{cases}
\, \rho(x) u_t= \Delta u^m +\, \rho(x) u^p & \text{in}\,\, B_R\times (0,T) \\
u=0 &\text{in}\,\, \partial B_R\times (0,T)\\
u =u_0 &\text{in}\,\, B_R\times \{0\}\,,
\end{cases}
\end{equation}
here $B_R$ denotes the Euclidean ball with radius $R$ and centre in $O$.

We shall use the following Aronson-Benilan type estimate (see \cite{AB}; see also \cite[Proposition 2.3]{Sacks}).

\begin{proposition}\label{prop2a}
Let $1<p< m$, $u_0\in H_0^1(B_R) \cap L^{\infty}(B_R)$, $u_0\ge 0$. Let $u$ be the solution to problem \eqref{problemapalla2}. Then, for a.e. $t\in(0,T)$,
\begin{equation*}
-\Delta u^m(\cdot,t) \le \rho u^p(\cdot, t)+ \frac{\rho}{(m-1)t} u(\cdot,t) \quad \text{in}\,\,\,\mathfrak{D}'(B_R).
\end{equation*}
\end{proposition}

\begin{proof}[Proof of Theorem \ref{teoesistenza2}] The conclusion follows using step by step the same arguments given in the proof of Theorem \ref{teoesistenza}, since the necessary functional inequalities are being assumed.  We use Proposition \ref{prop2a} instead of \ref{prop2}. The last statement of the Theorem will be proved later on in Section \ref{sec4}
\end{proof}

In order to prove Theorem \ref{teoesistenza3} we adapt the strategy of \cite{Sacks} to the present case, so we shall be concise and limit ourselves to identifying the main steps and differences. Define
\[d\mu:=\rho(x) dx\,.\]
For any $R>0, k>0$, for any $v\in L_{\rho}^1(B_R)$, we set
\begin{equation*}
A_k:= \{x\in B_R:\,|v(x)|>k\}
\end{equation*}
and
\begin{equation*}
g(k):= \int_{B_R}|G_k(v)| \rho(x)\,dx\,,
\end{equation*}
where $G_k(v)$ has been defined in \eqref{21}.

\begin{lemma}\label{lemma1pesi}
Let $v\in L_{\rho}^1(B_R)$. Suppose that there exist $C>0$ and $s>1$ such that
\begin{equation*}
g(k)\le C\mu(A_k)^{s} \quad \text{for any}\,\,k\in \R^+.
\end{equation*}
Then $v\in L^{\infty}(B_R)$ and  $$\|v\|_{L^{\infty}(B_R)}\le C\left(\frac{s}{s-1}\right)^{s}\|\rho\|_{L^1(\mathbb R^N)}^{s-1}.$$
\end{lemma}
\begin{proof}
Arguing as in the proof of Lemma \ref{lemma1}, we integrate inequality \eqref{24b} between $0$ and $k$ and using the definition of $g$, we obtain
$$
g(k)^{1-\frac 1 s} \le \|v\|_{L^1_{\rho}(B_R)}^{1-\frac 1 s} - \frac{s-1}{s}\,C^{-\frac 1 s} k \quad \text{for any}\,\, k\in\R^+\,.
$$
Choose $$k=k_0=C^{\frac 1 s}\|v\|_{L^1_{\rho}(B_R)}^{1-\frac 1 s}\frac{s}{s-1},$$ and substitute it in the last inequality. Then we have
$$
\begin{aligned}
g(k_0)\le0 &\iff \int_{B_R}|G_{k_0}(v)|\,d\mu =0 \iff |G_{k_0}(v)|=0\\&
\iff  |v|\le k_0\iff|v|\le C^{\frac 1 s}\|v\|_{L_{\rho}^1(B_R)}^{1-\frac 1 s}\frac{s}{s-1}.
\end{aligned}
$$
Thanks to the assumption that $\rho\in L^1(\mathbb R^N)$, we can apply the weighted H\"{o}lder inequality to get
$$
\|v\|_{L^{\infty}(B_R)}\le \frac{s}{s-1}C^{\frac 1 s}\|v\|_{L^{\infty}(B_R)}^{1-\frac 1 s}\|\rho\|^{1-\frac 1 s}.
$$
Rearranging the terms in the previous inequality we obtain the thesis.
\end{proof}

\begin{lemma}\label{prop1-pesi}
Let $\rho$ satisfy \eqref{rho2} and $\rho\in L^1(\mathbb R^N)$. Let $f_1\in L^{m_1}_{\rho}(B_R)$ and $f_2\in L^{m_2}_{\rho}(B_R)$ where $$m_1>\frac N 2,\quad \,m_2 >\frac{N}{2}\,.$$  Assume that $v\in H_0^1(B_R)$, $v\ge 0$ is a subsolution to problem
\begin{equation*}
\begin{cases}
-\Delta v = \rho (f_1+f_2) & \text{in}\,\, B_R\\
v=0 &\text{on}\,\, \partial B_R
\end{cases}.
\end{equation*}
Then
\begin{equation}\label{eqa10pesi}
\|v\|_{L^{\infty}(B_R)}\le C_1\|f_1\|_{L_{\rho}^{m_1}(B_R)}+C_2\|f_2\|_{L_{\rho}^{m_2}(B_R)},
\end{equation}
where \begin{equation}\label{38b}
\begin{aligned}
&C_1= \frac{1}{C_s^2}\left(\frac{s}{s-1} \right)^s\,\|\rho\|_{L^1(\mathbb R^N)}^{\frac 2 N-\frac{1}{m_1}}\,,  \\
&C_2=\frac{1}{C_s^2}\left(\frac{s}{s-1} \right)^s\,\|\rho \|_{L^1(\mathbb R^N)}^{\frac 2 N-\frac{1}{m_2}}\,,
\end{aligned}
\end{equation}
with $s$ given by \eqref{37b}\,.
\end{lemma}
\begin{remark}
\rm If in Lemma \ref{prop1-pesi} we further assume that there exists a constant $k_0>0$ such that
\[\|f_1 \|_{L_{\rho}^{m_1}(B_R)}\leq k_0, \quad \|f_2 \|_{L_{\rho}^{m_2}(B_R)}\leq k_0 \quad \textrm{ for all }\,\, R>0, \]
then from \eqref{eqa10pesi}, we infer that the bound from above on $\|v\|_{L^{\infty}(B_R)}$ is independent of $R$. This fact will
have a key role in the proof of global existence for problem \eqref{problema2}.
\end{remark}

\begin{proof}[Proof of Lemma \ref{prop1-pesi}] By arguing as in the proof of Proposition \ref{prop1}, we get
\[
\int_{B_R}|G_k(v)|\, d\mu \le\frac{1}{C_s^2}\left[\|f_1\|_{L_{\rho}^{m_1}}\mu(A_k)^{1+\frac{2}{N}-\frac{1}{m_1}}+\|f_2\|_{L_{\rho}^{m_2}}\mu(A_k)^{1+\frac{2}{N}-\frac{1}{m_2}}\right].
\]
Thus
$$
\int_{B_R}|G_k(v)|\, d\mu \le\frac{1}{C_s^2}\,\mu(A_k)^{1+\frac{2}{N}-\frac{1}{l}}\left[\|f_1\|_{L_{\rho}^{m_1}}\|\rho\|_{L^1(\mathbb R^N)}^{\frac{1}{l}-\frac{1}{m_1}}+\|f_2\|_{L_{\rho}^{m_2}}\|\rho\|_{L^1(\mathbb R^N)}^{\frac{1}{l}-\frac{1}{m_2}}\right].
$$
Now, defining
$$
\bar C=\frac{1}{C_s^2}\left[\|f_1\|_{L^{m_1}(B_R)}\|\rho\|_{L^1(\mathbb R^N)}^{\frac{1}{l}-\frac{1}{m_1}}+\|f_2\|_{L^{m_2}(B_R)}\|\rho\|_{L^1(\mathbb R^N)}^{\frac{1}{l}-\frac{1}{m_2}}\right]\,,
$$
the last inequality is equivalent to
$$
\int_{B_R}|G_k(v)|\,d\mu \le \bar C\,\mu(A_k)^{s}\,, \quad \text{for any}\,\,k\in\R^+\,,
$$
where $s$ has been defined in \eqref{37b}. Hence, it is possible to apply Lemma \ref{lemma1pesi}.
By using the definitions of $C_1$ and $C_2$ in \eqref{38b}, we thus have
$$
\|v\|_{L^{\infty}(B_R)}\le C_1\,\|f_1\|_{L_{\rho}^{m_1}(B_R)}+C_2\,\|f_2\|_{L_{\rho}^{m_2}(B_R)}\,.
$$

\end{proof}

\begin{proposition}\label{teo2pesi}
Let $1<p<m$, $R>0, u_0\in  L^{\infty}(B_R)$, $u_0\ge 0$. Let $u$ be the solution to problem \eqref{problemapalla2}. Let  inequality \eqref{S-pesi} hold.
Then there exists $C=C(p, m, N, C_s, \|\rho\|_{L^1(\mathbb R^N)})>0$ such that, for all $t>0$,
\begin{equation*}
\|u(t)\|_{L^{\infty}(B_R)}\le C \left[1+\left(\frac{1}{(m-1)t}\right)^{\frac{1}{m-1}}  \right].
\end{equation*}
\end{proposition}

\begin{proof}
We proceed as in the proof of Proposition \ref{teo2}, up to inequality \eqref{49a}. Thanks to the fact that $\rho\in L^1(\mathbb R^N)$, we can apply to \eqref{50a} the thesis of Lemma \ref{prop1-pesi}. Thus we obtain
\begin{equation}\label{eq4pesi}
\|w\|_{L^{\infty}(B_R)}^m\le C_1\|w^p\|_{L_{\rho}^{r_1}(B_R)}+\gamma C_2\|w\|_{L_{\rho}^{r_2}(B_R)}.
\end{equation}
Now the constants are
$$
\begin{aligned}
&\alpha_1=\frac{m}{p-\frac{q}{r_1}}; \\
&\alpha_2=\frac{m}{1-\frac{q}{r_2}}; \\
&\varepsilon_1 \text{ such that}\,\,\delta_1=\frac{1}{4C_1}; \\
&\varepsilon_2 \text{ such that}\,\,\delta_2=\frac{1}{4\gamma C_2}.
\end{aligned}
$$
Plugging \eqref{49} and \eqref{49a} into \eqref{eq4pesi} we obtain
\begin{equation}\label{eq5pesi}
\begin{aligned}
\|w\|_{L^{\infty}(B_R)}^{m} &\le C_1\|w^p\|_{L_{\rho}^{r_1}(B_R)}+\gamma C_2\|w\|_{L_{\rho}^{r_2}(B_R)}  \\
&\le C_1\left[\delta_1\left\|w\right\|_{L^{\infty}(B_R)}^{m} + \frac{\eta(\alpha_1)}{\delta_1^{\frac{1}{\alpha_1-1}}}\left\|w\right\|_{L_{\rho}^{q}(B_R)}^{\frac{mq}{r_1}\frac{1}{m-p+q/r_1}} \right] \\ &+ \gamma C_2 \left[ \delta_2\left\|w\right\|_{L^{\infty}(B_R)}^{m} + \frac{\eta(\alpha_2)}{\delta_2^{\frac{1}{\alpha_2-1}}}\left\|w\right\|_{L_{\rho}^{q}(B_R)}^{\frac{mq}{r_2}\frac{1}{m-1+q/r_2}} \right] .
\end{aligned}
\end{equation}
Inequality \eqref{eq5pesi} can be rewritten as
\begin{equation*}
\begin{aligned}
\|w\|_{L^{\infty}(B_R)} &\le \left[2\eta(\alpha_1)\left(4C_1^{\alpha_1}\right)^{\frac{1}{\alpha_1-1}}\right]^{\frac{1}{m}}\left\|w\right\|_{L_{\rho}^{q}(B_R)}^{\frac{q}{r_1}\frac{1}{m-p+q/r_1}}\\&+\left[2 \eta(\alpha_2)\left(4\gamma^{\alpha_2}C_2^{\alpha_2}\right)^{\frac{1}{\alpha_2-1}}\right]^{\frac{1}{m}}\left\|w\right\|_{L_{\rho}^{q}(B_R)}^{\frac{q}{r_2}\frac{1}{m-1+q/r_2}}.
\end{aligned}
\end{equation*}
Computing the limits as $r_1\longrightarrow \infty$ and $r_2\longrightarrow \infty$ we have
$$
\begin{aligned}
&\eta(\alpha_1)\longrightarrow\left[\frac{p}{m}\right]^{\frac{p}{m-p}}\left\{1-\frac{p}{m}\right\};\\
&\eta(\alpha_2)\longrightarrow\left[\frac{1}{m}\right]^{\frac{1}{m-1}}\left\{1-\frac{1}{m}\right\};\\
&\left\|w\right\|_{L_{\rho}^{q}(B_R)}^{\frac{q}{r_1}\frac{1}{(m-p+q/r_1)}} \longrightarrow 1;\\
&\left\|w\right\|_{L_{\rho}^{q}(B_R)}^{\frac{q}{r_2}\frac{1}{(m-1+q/r_2)}}\longrightarrow 1.
\end{aligned}
$$
Moreover we define
$$
\begin{aligned}
&\Gamma_1:=\left[2\left(\frac{p}{m}\right)^{\frac{p}{m-p}} \left(1-\frac{p}{m}\right)\right]^{\frac 1 m}4^{\frac{mp}{m-p}}C_1^{\frac{mp}{m-p}};\\
&\Gamma_2:=\left[2\left(\frac{1}{m}\right)^{\frac{1}{m-1}} \left(1-\frac{1}{m}\right)\right]^{\frac 1 m}4^{\frac{m}{m-1}}C_1^{\frac{m}{m-1}};\\
&C:=\max\{\Gamma_1\,\,,\Gamma_2\}
\end{aligned}
$$
and notice that, by the above construction, the thesis follows with this choice of $C$.
\end{proof}

\bigskip

\begin{proof}[Proof of Theorem \ref{teoesistenza3}] The conclusion follows by the same arguments as in the proof of Theorem \ref{teoesistenza}. However, some minor differences are in order. We replace Proposition \ref{teo2} by Proposition \ref{teo2pesi}. Moreover, since $u_0\in L^1_{\rho}(\mathbb R^N)$, the family of functions $\{u_{0h}\}$ is as follows:
\begin{equation*}
\begin{aligned}
&u_{0,h}\in L^{\infty}(\mathbb R^N)\cap C_c^{\infty}(\mathbb R^N) \,\,\,\text{for all} \,\,h\ge 0, \\
&u_{0,h}\ge 0 \,\,\,\text{for all} \,\,h\ge 0, \\
&u_{0, h_1}\leq u_{0, h_2}\,\,\,\text{for any } h_1<h_2,  \\
&u_{0,h}\longrightarrow u_0 \,\,\, \text{in}\,\, L^1_{\rho}(\mathbb R^N)\quad \textrm{ as }\, h\to +\infty\,.
\end{aligned}
\end{equation*}
Furthermore, instead of \eqref{6}, \eqref{10}, \eqref{12}, \eqref{15}, we use the following. By standard arguments  (see, e.g. proof of \cite[Proposition 2.5-(i)]{Sacks}) we have that
\[\|u^R_{h,k}(t)\|_{L^1_{\rho}(B_R)}\leq C \|u_{0h}\|_{L^1_{\rho}(B_R)}\quad \textrm{ for all}\,\, t>0\,, \]
for some positive constant $C=C(p, m, N, \|\rho\|_{L^1(\mathbb R^N)}),$ and, for any $\varepsilon\in (0, m-p)$,
\[\int_0^1\int_{B_R} (u^R_{h, k})^{p+\varepsilon}\rho(x) dx dt \leq \tilde C\,,\]
for some positive constant $\tilde C=\tilde C(p, m, N, \|\rho\|_{L^1(\mathbb R^N)}, \|u_0\|_{L^1_{\rho}(\mathbb R^N)}).$ Hence, after having passed to the limit as $k\to +\infty, R\to +\infty, h\to +\infty$,
 for any $T>0, \varphi \in C_c^{\infty}(\mathbb R^N\times(0,T))$ such that $\varphi(x,T)=0$ for every $x\in \mathbb R^N$,
we have that
\[\int_0^T\int_{\mathbb R^N} u^{p+\varepsilon}\rho(x)\varphi\, dx dt\leq C\,.\]
Therefore, \eqref{3a} holds.

\end{proof}

\subsection{End of proof of Theorem \ref{teoesistenza2}: an example of complete blowup in infinite time}\label{sec4}

We recall that we are assuming $m>1$ and $1<p<m$. Let us set $r:=|x|$. We now construct a subsolution to equation
\begin{equation}\label{91b}
\rho\, u_t= \Delta u^m +\rho\, u^p\quad  \text{in}\,\, \mathbb R^N\times (0,T)\,,
\end{equation}
under the hypothesis that there exist $k_1$ and $k_2$ with $k_2\geq k_1>0$ such that
\begin{equation}\label{95b}
k_1r^2\le \frac{1}{\rho(x)}\le k_2r^2\quad \text{for any}\,\,\,x\in\R^N\setminus B_e.
\end{equation}
Moreover, due to the running assumptions on the weight there exist positive constants $\rho_1,\rho_2$ such that
\begin{equation}\label{96b}
\rho_1\le \frac{1}{\rho(x)}\le \rho_2\quad \text{for any}\,\,\,x\in B_e\,.
\end{equation}

Let
\[\mathfrak{s}(x):=\begin{cases}
\log(|x|)  &\quad \text{if}\quad  x\in \R^N\setminus B_e, \\
& \\
\dfrac{|x|^2+e^2}{2e^2} &\quad\text{if}\quad  x\in B_e\,.
\end{cases}\]

The requested statements will follow from the following result.

\begin{proposition}\label{teosubsolutioncritical}
Let assumption \eqref{rho2}, \eqref{95b} and \eqref{96b} be satisfied, and $1<p<m.$ If the initial datum $u_0$ is smooth, compactly supported and large enough, then problem \eqref{problema2} has a solution $u(t)\in L^\infty(\mathbb R^N)$ for any $t\in (0,\infty)$ that blows up in infinite time, in the sense that
\begin{equation}\label{eq22}
\lim_{t\to+\infty}u(x, t)= +\infty \ \ \ \forall x\in {\mathbb R}^N.
\end{equation}
More precisely, if $C>0$, $a>0$, $\alpha>0$, $\beta>0$, $T>0$ verify
\begin{equation}\label{98b}
0<T^{-\beta}<\frac{a}{2}.
\end{equation}
\begin{equation}\label{alphabeta}
0<\alpha<\frac{1}{m-1}\,,\quad\quad \beta=\frac{\alpha(m-1)+1}{2}\,,
\end{equation}
and
\begin{equation*}
u_0(x)\ge CT^{\alpha}\left[1-\frac{\mathfrak{s}(x)}{a}\,T^{-\beta}\right]^{\frac{1}{m-1}}_{+}\,, \quad \text{for any}\,\, x\in \R^N\,,
\end{equation*}
then the solution $u$ of problem \eqref{problema2} satisfies \eqref{eq22} and the bound from below
\begin{equation*} 
u(x,t) \ge C (T+t)^{\alpha}\left [1- \frac{\mathfrak{s}(x)}{a}\, (T+t)^{-\beta} \right ]_{+}^{\frac{1}{m-1}}, \,\, \text{for any}\,\, (x,t) \in \R^N\times(0,+\infty)\,.
\end{equation*}
\end{proposition}

\begin{proof}
We construct a suitable subsolution of \eqref{91b}. Define, for all $(x,t)\in \R^N$,
\begin{equation*}
w(x,t)\equiv w(r(x),t) :=
\begin{cases}
 u(x,t) \quad \text{in } [\R^N \setminus B_{e}] \times (0,T), \\
v(x,t) \quad \text{in } B_{e} \times (0,T),
\end{cases}
\end{equation*}
where
\begin{equation*}
{{u}}(x,t)\equiv {u}(r(x),t):=C(T+t)^{\alpha}\left [1-\frac{\log(r)}{a}(T+t)^{-\beta}\right]_{+}^{\frac{1}{m-1}},
\end{equation*}
and
\begin{equation*}
v(x,t) \equiv v(r(x),t):= C(T+t)^{\alpha} \left [ 1-\frac{r^2+e^2}{2e^2} \frac{(T+t)^{-\beta}}{a} \right ]^{\frac{1}{m-1}}_{+}\,.
\end{equation*}
Moreover, let
$$
F(r,t):= 1-\frac{\log(r)}{a}(T+t)^{-\beta}\,,
$$
$$
G(r,t):= 1-\frac{r^2+e^2}{2e^2}\frac{(T+t)^{-\beta}}{a}\,.
$$
and define
$$
\textit{D}_1:=\left \{ (x,t) \in (\R^N \setminus B_e) \times (0,T)\,\, |\,\, 0<F(r,t)<1 \right \}.
$$
For any $(x,t) \in D_1$, we have:
\begin{equation*}
\begin{aligned}
{u}_t
=C\alpha(T+t)^{\alpha-1}  F^{\frac{1}{m-1}} - C\beta(T+t)^{\alpha-1} \frac{1}{m-1}F^{\frac{1}{m-1}} + C\beta(T+t)^{\alpha-1} \frac{1}{m-1} F^{\frac{1}{m-1}-1}. \\
\end{aligned}
\end{equation*}
\begin{equation*}
({u}^m)_r=-\frac{C^m}{a} (T+t)^{m\alpha} \frac{m}{m-1} F^{\frac{1}{m-1}} \frac{1}{r}(T+t)^{-\beta}.
\end{equation*}
\begin{equation*}
\begin{aligned}
({u}^m)_{rr}&=\frac{C^m}{a} (T+t)^{m\alpha} \frac{m}{m-1} F^{\frac{1}{m-1}}  \frac{(T+t)^{-\beta}}{r^2}\\&+ \frac{C^m}{a^2} (T+t)^{m\alpha} \frac{m}{(m-1)^2} F^{\frac{1}{m-1}-1} \frac{(T+t)^{-2\beta}}{r^2}.
\end{aligned}
\end{equation*}
Due to \eqref{98b}
$$0<G(r,t)<1 \quad \text{for all}\,\,\,(x,t)\in B_e\times(0,+\infty).$$ For any $(x,t) \in B_e \times (0,T)$, we have:
\begin{equation*}
\begin{aligned}
v_t
=C\alpha(T+t)^{\alpha-1}G^{\frac{1}{m-1}} - C\beta(T+t)^{\alpha-1} \frac{1}{m-1} G^{\frac{1}{m-1}} + C\beta(T+t)^{\alpha-1} \frac{1}{m-1}G^{\frac{1}{m-1}-1}\,.
\end{aligned}
\end{equation*}
\begin{equation*}
(v^m)_r=-\frac{C^m}{a} (T+t)^{m\alpha} \frac{m}{m-1} G^{\frac{1}{m-1}} \frac{r}{e^2}(T+t)^{-\beta}\,.
\end{equation*}
\begin{equation*}
(v^m)_{rr}=-\frac{C^m}{a} (T+t)^{m\alpha} \frac{m}{m-1} G^{\frac{1}{m-1}} \frac {(T+t)^{-\beta}}{e^2}+ \frac{C^m}{a^2} (T+t)^{m\alpha} \frac{m}{(m-1)^2} G^{\frac{1}{m-1}-1} (T+t)^{-2\beta}\frac{r^2}{e^4}\,.
\end{equation*}

For every $(x,t)\in D_1$, by the previous computations we have
\begin{equation}\label{104}
\begin{aligned}
u_t-&\frac{1}{\rho}\Delta u^m-u^p\\
&=C\alpha(T+t)^{\alpha-1}  F^{\frac{1}{m-1}} - C\beta(T+t)^{\alpha-1} \frac{1}{m-1}F^{\frac{1}{m-1}}+ C\beta(T+t)^{\alpha-1} \frac{1}{m-1} F^{\frac{1}{m-1}-1}\\
&+\frac{1}{\rho}\left\{-\frac{C^m}{a} (T+t)^{m\alpha-\beta} \frac{m}{m-1} F^{\frac{1}{m-1}}\frac{1}{r^2}- \frac{C^m}{a^2} (T+t)^{m\alpha-2\beta} \frac{m}{(m-1)^2} F^{\frac{1}{m-1}-1} \frac{1}{r^2}\right.\\
&\left. +(N-1)\frac{C^m}{a} (T+t)^{m\alpha-\beta} \frac{m}{m-1} F^{\frac{1}{m-1}} \frac{1}{r^2}\right\} -C^p (T+t)^{p\alpha}{F^{\frac p{m-1}}}.
\end{aligned}
\end{equation}
Thanks to \eqref{95b}, \eqref{104} becomes, for every $(x,t)\in D_1$
\begin{equation*}
\begin{aligned}
u_t-&\frac{1}{\rho}\Delta u^m-u^p\\
&\le CF^{\frac{1}{m-1}-1}\left\{F\left[\alpha(T+t)^{\alpha-1}-\frac{\beta}{m-1}(T+t)^{\alpha-1}+(N-2)k_2\frac{C^{m-1}}{a}\frac{m}{m-1}(T+t)^{m\alpha-\beta}\right]\right.\\
&\left.+\frac{\beta}{m-1}(T+t)^{\alpha-1}-\frac{C^{m-1}}{a^2}\frac{m}{(m-1)^2}k_1(T+t)^{m\alpha-2\beta}-C^{p-1}(T+t)^{p\alpha}F^{\frac{p+m-2}{m-1}}\right\}\\
& \leq CF^{\frac{1}{m-1}-1}\left\{\sigma(t)F-\delta(t)-\gamma(t)F^{\frac{p+m-2}{m-1}}\right\}
\end{aligned}
\end{equation*}
where
\begin{equation*}
\varphi(F):=\sigma(t)F-\delta(t)-\gamma(t)F^{\frac{p+m-2}{m-1}},
\end{equation*}
with
\begin{equation*}
\begin{aligned}
& \sigma(t) = \left [ \alpha-\frac{\beta}{m-1}\right](T+t)^{\alpha-1}+\frac{C^{m-1}}{a} \frac{m}{m-1}k_2\left(N-2\right)(T+t)^{m\alpha-\beta}\,,\\
& \delta(t) = -\frac{\beta}{m-1} (T+t)^{\alpha-1}+ \frac{C^{m-1}}{a^2} \frac{m}{(m-1)^2}k_1(T+t)^{m\alpha-2\beta}\,,  \\
& \gamma(t)=C^{p-1} (T+t)^{p\alpha }\,, \\
\end{aligned}
\end{equation*}

Our goal is to find suitable $C>0$, $a>0$, such that
$$
\varphi(F) \le 0\,, \quad \text{for all}\,\,  F \in (0,1)\,.
$$
To this aim, we impose that
$$
\sup_{F\in (0,1)}\varphi(F)=\max_{F\in (0,1)}\varphi(F)= \varphi (F_0)\leq 0\,,
$$
for some $F_0\in(0,1)$. We have
$$
\begin{aligned}
\frac{d \varphi}{dF}=0
&\iff \sigma(t) - \frac{p+m-2}{m-1} \gamma(t) F^{\frac{p-1}{m-1}} =0 \\ & \iff F_0= \left [\frac{m-1}{p+m-2} \frac{\sigma(t)}{\gamma(t)} \right ]^{\frac{m-1}{p-1}}\,.
\end{aligned}
$$
Then
$$
\varphi(F_0)= K \dfrac{\sigma(t)^{\frac{p+m-2}{p-1}}}{\gamma(t)^{\frac{m-1}{p-1}}} - \delta(t)
$$
where $K=\left(\frac{m-1}{p+m-2}\right)^{\frac{m-1}{p-1}}-\left(\frac{m-1}{p+m-2}\right)^{\frac{p+m-2}{p-1}}>0 $. The two conditions we must verify are
\begin{equation}\label{107}
\begin{aligned}
K[\sigma(t)]^{\frac{p+m-2}{p-1}} \le \delta(t) \gamma(t)^{\frac{m-1}{p-1}}\,,\ \
(m-1) \sigma(t) \le (p+m-2) \gamma(t) \,.
\end{aligned}
\end{equation}
Observe that, thanks to the choice in \eqref{alphabeta} and by choosing $$\frac{C^{m-1}}{a}  \ge 2\beta\,\frac{(m-1)}{m}\frac{1}{k_1},$$ we have
$$
\begin{aligned}
& \sigma(t)\le \frac{C^{m-1}}{a} \frac{m}{m-1}k_2\left(N-2\right)(T+t)^{m\alpha-\beta}\,,\\
& \delta(t)\ge  \frac{C^{m-1}}{2a^2} \frac{m}{(m-1)^2}k_1(T+t)^{m\alpha-2\beta}\,
\end{aligned}
$$
and conditions in \eqref{107} follow. So far, we have proved that
\begin{equation*}
{u}_t-\frac{1}{\rho(x)}\Delta({u}^m)-{u}^p \le 0 \quad \text{ in } D_1\,.
\end{equation*}
Furthermore, since ${u}^m\in C^1([\R^N\setminus B_e]\times[0,T))$ 
it follows that $ u$ is a subsolution to equation \eqref{91b} in $ [\R^N \setminus B_e]\times (0,T)$.
Now, we consider equation \eqref{91b} in $B_e\times(0,T)$. We observe that, due to condition \eqref{98b},
\begin{equation}\label{99}
\frac{1}{2}<G<1\,\,\,\text{for all}\,\,(x,t)\in B_e\times(0,T).
\end{equation}
Similarly to the previous computation we obtain, for all $(x,t)\in B_e\times(0,T)$:
\begin{equation*}
v_t-\frac{1}{\rho}\Delta v^m-v^p\le CG^{\frac{1}{m-1}-1}\psi(G)\,,
\end{equation*}
where
$$
\psi(G):=\sigma_0G-\delta_0-\gamma G^{\frac{p+m-2}{m-1}}\,,
$$
with
\begin{equation*}
\begin{aligned}
&  \sigma_0(t) = \left [ \alpha-\frac{\beta}{m-1}\right](T+t)^{\alpha-1}+\rho_2\frac{N}{e^2} \frac{m}{m-1}\frac{C^{m-1}}{a} (T+t)^{m\alpha-\beta} \,,\\
&  \delta_0(t) = -\frac{\beta}{m-1} (T+t)^{\alpha-1}\,\\
& \gamma(t)=C^{p-1} (T+t)^{p\alpha }\,.
\end{aligned}
\end{equation*}

Due to \eqref{99}, $v$ is a subsolution of \eqref{91b} for every $(x,t)\in B_e\times(0,T)$, if
$$
2^{\frac{p+m-2}{m-1}}\left(\sigma_0-\delta_0\right)\le\gamma\,.
$$
This last inequality is always verified thanks to \eqref{alphabeta}. Hence we have proved that
\begin{equation*}
v_t-\frac{1}{\rho(x)}\Delta(v^m)-v^p \le 0 \quad \text{ in }\,\, B_e\times(0,T)\,,
\end{equation*}
Moreover, $w^m \in C^1(\R^N \times [0,T))$, indeed,
$$
({u}^m)_r = (v^m)_r  = -C^m \zeta(t)^m \frac{m}{m-1} \frac{1}{e}\frac{\eta(t)}{a} \left [ 1- \frac{\eta(t)}{a} \right ]_+^{\frac{1}{m-1}} \quad \text{in}\,\, \partial B_e\times (0,T)\,.
$$
Hence,
$w$ is a subsolution to equation \eqref{91b} in $\R^N\times(0,T)$.

\end{proof}

\par\bigskip\noindent
\textbf{Acknowledgments. } The first and third author are partially supported by the PRIN project 201758MTR2 ``Direct and inverse problems for partial differential equations: theoretical aspects and applications'' (Italy). All authors are members of the Gruppo Nazionale per l'Analisi Ma\-te\-ma\-ti\-ca, la Probabilit\`a e le loro Applicazioni (GNAMPA) of the Istituto Nazionale di Alta Ma\-te\-ma\-ti\-ca (INdAM). The third author is partially supported by GNAMPA Projects 2019, 2020.

%
%

%


\end{document}